\documentclass[letter,article,oneside,11pt]{memoir}

\usepackage[english]{babel}
\usepackage{amsfonts}
\usepackage{mathrsfs}
\usepackage{amssymb}
\usepackage{amsmath}
\usepackage{enumitem}\setlist{nolistsep}
\usepackage{amsthm}
\usepackage{hyperref}
\usepackage[capitalize]{cleveref}
\usepackage{tikz-cd}
\usepackage{todonotes}

\RequirePackage[backend=bibtex,style=alphabetic,
url=false,doi=false,maxbibnames=50,maxcitenames=50,maxnames=50]{biblatex}
\defbibheading{bibliography}[Bibliography]{\section*{#1}\markboth{#1}{#1}}
\settrims{0pt}{0pt}
\setheadfoot{0.5in}{0.5in}
\setulmarginsandblock{1.0in}{1.0in}{*}
\setlrmarginsandblock{1.0in}{1.0in}{*}
\checkandfixthelayout

\renewcommand{\printtoctitle}[1]{\large{\textbf{Contents}}}

\OnehalfSpacing
\setsecheadstyle{\large\bfseries}

\setsecnumdepth{subsubsection}
\setsecnumformat{\csname the#1\endcsname. }
\setsubsecheadstyle{\normalfont\bfseries}

\counterwithout{section}{chapter}
\counterwithout{equation}{chapter}

\newtheoremstyle{plain}{3mm}{3mm}{\slshape}{}{\bfseries}{.}{.5em}{}
\newtheoremstyle{definition}{2mm}{2mm}{}{}{\bfseries}{.}{.5em}{}

\theoremstyle{plain}
\newtheorem{theorem}{Theorem}[section]

\newtheorem{proposition}[theorem]{Proposition}
\newtheorem{corollary}[theorem]{Corollary}
\newtheorem{lemma}[theorem]{Lemma}

\theoremstyle{definition}
\newtheorem{example}[theorem]{Example}
\newtheorem{question}[theorem]{Question}

\newtheorem{definition}[theorem]{Definition}

\theoremstyle{plain}
\newtheorem*{namedthm}{\namedthmname}
\newcounter{namedthm}
	

\usepackage{xcolor}

\definecolor{Scarlet}{rgb}{0.78, 0.11, 0.0}
\definecolor{Blue}{rgb}{0.0, 0.42, 0.47}
\definecolor{Green}{rgb}{0.39, 0.71 ,0.0}

\hypersetup{citecolor = Scarlet,colorlinks,
			linkcolor = blue,
			urlcolor = Scarlet}

\newcommand{\N}{\mathbb{N}}
\newcommand{\Z}{\mathbb{Z}}
\newcommand{\R}{\mathbb{R}}
\newcommand{\T}{\mathbb{T}}
\newcommand{\C}{\mathbb{C}}

\newcommand{\define}[1]{\textbf{#1}}
\renewcommand{\epsilon}{\varepsilon}
\renewcommand{\leq}{\leqslant}
\renewcommand{\geq}{\geqslant}
\DeclareMathOperator{\id}{id}

\newcommand{\E}{\mathbb{E}}

\newcommand{\1}{1}

\newcommand{\nbar}{|\!|}
\newcommand{\norm}[2]{\nbar {#1} \nbar_{#2}}
\newcommand{\bnbar}{\Bigg\vert\!\Bigg\vert}
\newcommand{\bnorm}[2]{\bnbar {#1} \bnbar_{#2}}
\newcommand{\Aut}{\mathsf{Aut}}
\newcommand{\cont}{\mathsf{C}}
\newcommand{\kron}{\mathsf{K}}
\newcommand{\ap}{\mathsf{AP}}
\newcommand{\wm}{\mathsf{WM}}
\newcommand{\intd}{~\mathsf{d}}
\renewcommand{\d}{~\mathsf{d}}
\newcommand{\joinings}{\mathcal{J}}
\newcommand{\ergjoinings}{\mathcal{J}_\mathsf{e}}
\DeclareMathOperator{\lp}{L}
\newcommand{\bilin}[2]{\langle {#1}, {#2} \rangle}

\newcommand{\pionethree}{\pi_{1,3}}

\newcommand{\Eig}{\mathsf{Eig}}
\newcommand{\Folner}{F\o{}lner}
\newcommand{\expec}{\mathbb{E}}
\newcommand{\condex}[2]{\expec({#1}\vert{#2})}
\newcommand{\WW}{{\mathbf W}'}

\newcommand{\change}[1]{{#1}}

\title{Disjointness for measurably distal group actions and applications}
\author{Joel Moreira \and Florian K. Richter \and Donald Robertson}

\begin{document}

\allowdisplaybreaks

\date{\small \today}

\maketitle

\begin{abstract}
We generalize Berg's notion of quasi-disjointness to actions of countable groups and prove that every measurably distal system is quasi-disjoint from every measure preserving system.
As a corollary we obtain easy to check necessary and sufficient conditions for two systems to be disjoint, provided one of them is measurably distal.
We also obtain a Wiener--Wintner type theorem for countable amenable groups with distal weights and applications to weighted multiple ergodic averages and multiple recurrence.
\end{abstract}

\tableofcontents*

\section{Introduction}

By a $\mathbb{Z}$ \define{system} we mean a tuple $\mathbf{X} = (X,T,\mu_X)$ where $X$ is a compact, metric space, $T$ is a continuous action of $\mathbb{Z}$ on $X$ and $\mu_X$ is a Borel probability measure on $X$ that is $T$ invariant.
Ergodic $\mathbb{Z}$ systems $(X,T,\mu_X)$ and $(Y,S,\mu_Y)$ are \define{disjoint} if $\mu_X \otimes \mu_Y$ is the only probability measure on $X \times Y$ that has $\mu_X$ and $\mu_Y$ as its marginals and is invariant under the diagonal action $(T \times S)^n = T^n \times S^n$.
(Any probability measure on $X \times Y$ with these two properties is called a \define{joining} of the two $\mathbb{Z}$ systems.)
The notion of disjointness -- introduced in Furstenberg's seminal paper \cite{Furstenberg67} -- is an extreme form of non-isomorphism.
In particular, if systems have a non-trivial factor in common then they cannot be disjoint.
Furstenberg asked whether the converse is true.
Rudolph~\cite{MR555301} answered this question by producing (from his construction in the same paper of a $\mathbb{Z}$ system with minimal self-joinings) two $\mathbb{Z}$ systems that are not disjoint and yet share no common factor.

Perhaps motivated by Furstenberg's question, Berg~\cite{Berg71,Berg72} considered the case when one of the systems is measurably distal.
Recall that a $\mathbb{Z}$ system is \define{measurably distal} if it belongs to the smallest class of $\mathbb{Z}$ systems that contains the trivial system and is closed under factors, group extensions and inverse limits - see Section~\ref{sec:prelims} for the definitions of these notions.
That the above definition of measurably distal is equivalent to Parry's original definition \cite{Parry68} in terms of separating sieves was proved by Zimmer~\cite{Zimmer76} (cf.\ Subsection \ref{subsec:prelim-distal} below).

To describe Berg's result, recall that the \define{Kronecker factor} of an ergodic $\mathbb{Z}$ system is the largest factor of the system that is isomorphic to a rotation on a compact abelian group.
Berg proved that disjointness of the Kronecker factors of two ergodic $\mathbb{Z}$ systems is equivalent to both disjointness and the absence of a common factor when one of the systems is measurably distal.
\change{We say that systems $\mathbf{X}$ and $\mathbf{Y}$ are \define{Kronecker disjoint} if their Kronecker factors are disjoint.}

\begin{theorem}[Berg, \cite{Berg71,Berg72}]
\label{thm:bergWeylDisjoint}
Let $\mathbf{X}$ be an ergodic and measurably distal $\Z$ system and let $\mathbf{Y}$ be an ergodic $\Z$ system. The following are equivalent:
\begin{enumerate}	
[label=(\roman{enumi}),ref=(C\roman{enumi}),leftmargin=*]
\item
$\mathbf{X}$ and $\mathbf{Y}$ are disjoint;
\item
$\mathbf{X}$ and $\mathbf{Y}$ are Kronecker disjoint;
\item
$\mathbf{X}$ and $\mathbf{Y}$ have no non-trivial common factor.
\end{enumerate}
\end{theorem}


The key ingredient in the proof of Theorem~\ref{thm:bergWeylDisjoint} is a weakening of the notion of disjointness that is shown to be preserved by group extensions, factors and inverse limits.
The definition of this weakened property (called ``quasi-disjointness'' in \cite{Berg71}) is as follows.
Given ergodic $\mathbb{Z}$ systems $\mathbf{X}$ and $\mathbf{Y}$ let $\alpha$ and $\beta$ be the factor maps from $\mathbf{X}$ and $\mathbf{Y}$ respectively to their maximal common Kronecker factor $\kron(\mathbf{X},\mathbf{Y})$.
The $\mathbb{Z}$ systems $\mathbf{X}$ and $\mathbf{Y}$ are quasi-disjoint if the property
\begin{itemize}[leftmargin=2cm]
\item[(BQD)]
for almost every $k$ in $\kron(\mathbf{X},\mathbf{Y})$ there is exactly one joining of the systems $\mathbf{X}$ and $\mathbf{Y}$ giving full measure to $\gamma^{-1}(k)$
\end{itemize}
holds, where $\gamma(x,y)=\alpha(x)-\beta(y)$.
The main results in \cite{Berg71,Berg72} imply that $\mathbf{X}$ and $\mathbf{Y}$ satisfy (BQD) whenever $\mathbf{X}$ is an ergodic and measurably distal $\mathbb{Z}$ system and $\mathbf{Y}$ is an ergodic $\mathbb{Z}$ system.

In this paper we introduce a new definition of quasi-disjointness that applies to measure preserving actions of any countable group $G$.
To describe it we recall the following notions.
A $G$ \define{system} is a tuple $(X,T,\mu_X)$ where $X$ is a compact, metric space, $T$ is a continuous left action of $G$ on $X$ and $\mu_X$ is a Borel probability measure on $X$ that is $T$ invariant.
The \define{Kronecker factor} of a $G$ system $\mathbf{X}$ is the factor $\kron \mathbf{X}$ corresponding to the subspace of $\lp^2(X,\mu_X)$ spanned by functions $f$ with the property that $\{ f \circ T^g : g \in G \}$ has compact closure.
Disjointness of $G$ systems is defined just as for $\mathbb{Z}$ systems.

Although it is the case for ergodic $\mathbb{Z}$ systems, the Kronecker factor of an ergodic $G$ system cannot generally be modeled by a rotation on a compact, abelian group.
Thus it is not clear how to modify (BQD) or Berg's proofs to apply to actions of more general groups.
We instead make the following definition, which is more general and easier to handle than (BQD).

\begin{definition}
\label{def:quasidisjoint}
Two $G$ systems $\mathbf{X} = (X,T,\mu_X)$ and $\mathbf{Y} = (Y,S,\mu_Y)$ are \define{quasi-disjoint} if the only joining of $\mathbf{X}$ and $\mathbf{Y}$ that projects to the product measure on the product $\kron\mathbf{X} \times \kron \mathbf{Y}$ of their Kronecker factors is the trivial joining $\mu_X \otimes \mu_Y$.
\end{definition}

Our first result (proved in \cref{sec:samequasidisjoint}) is that ergodic $\mathbb{Z}$ systems $\mathbf{X}$ and $\mathbf{Z}$ are quasi-disjoint according to \cref{def:quasidisjoint} if and only if they satisfy (BQD), justifying the use of the terminology ``quasi-disjoint''.

\begin{theorem}\label{thm_samequasidisjoint}
Ergodic $\mathbb{Z}$ systems $\mathbf{X}$ and $\mathbf{Y}$ are quasi-disjoint if and only if they satisfy (BQD).
\end{theorem}

Our main result is an extension of Berg's main results in \cite{Berg71,Berg72} to $G$ systems.
Recall that a $G$ system is \define{measurably distal} if it belongs to the smallest class of $G$ systems that is closed under factors, group extensions and inverse limits -- notions that are defined in Section~\ref{sec:prelims}.

\begin{theorem}
\label{thm:main-result}
If $G$ is a countable group and $\mathbf{X}$ is a measurably distal $G$ system, then $\mathbf{X}$ is quasi-disjoint from any other $G$ system $\mathbf{Y}$.
\end{theorem}

As a consequence of Theorem~\ref{thm:main-result} we obtain the following characterizations of disjointness from a measurably distal system.

\begin{corollary}
\label{1st_cor_of_main_result}
If $G$ is a countable group, $\mathbf{X}$ is a measurably distal $G$ system and $\mathbf{Y}$ is a $G$ system then the following are equivalent:
\begin{enumerate}	
[label=(\roman{enumi}),ref=(\roman{enumi}),leftmargin=*]
\item\label{cor:1.5-item-i}
$\mathbf{X}$ and $\mathbf{Y}$ are disjoint;
\item\label{cor:1.5-item-ii}
$\mathbf{X}$ and $\mathbf{Y}$ are Kronecker disjoint;
\item\label{cor:1.5-item-iii}
$\mathbf{X}$ and $\mathbf{Y}$ have no nontrivial common factor.
\end{enumerate}
If additionally $\mathbf{X}$ and $\mathbf{Y}$ are ergodic then \ref{cor:1.5-item-i} -- \ref{cor:1.5-item-iii} are also equivalent to
\begin{enumerate}	
[label=(\roman{enumi}),ref=(\roman{enumi}),leftmargin=*]
\setcounter{enumi}{3}
\item\label{cor:1.5-item-vi}
The product-system $\mathbf{X}\times\mathbf{Y}$ is ergodic.
\end{enumerate}
\end{corollary}

If $\mathbf{Y}$ is measurably distal and both $\mathbf{X}$ and $\mathbf{Y}$ are ergodic then, as pointed out to us by Glasner, the structure theory of measurably distal systems together with \cite[Theorem~3.30]{MR1958753} provide an alternative approach to proving \cref{1st_cor_of_main_result}.

We offer two applications of our results.
The first is a Wiener--Wintner type result with distal weights for measure preserving actions of countable amenable groups, proved in \cref{sec:wienerWintner}.
Recall that, when dealing with actions of amenable groups, one uses \Folner{} sequences to average orbits, where a \define{\Folner{} sequence} in a countable group $G$ is a sequence $N \mapsto \Phi_N$ of finite, non-empty subsets of $G$ such that
\[
\frac{|\Phi_N \cap g^{-1} \Phi_N|}{|\Phi_N|} \to 1
\]
for all $g$ in $G$.
Lindenstrauss~\cite{MR1865397} proved that the pointwise ergodic theorem holds for actions of amenable groups along \define{tempered} \Folner{} sequences -- those for which there is $C > 0$ with
\[
\left| \bigcup_{K < N} \Phi_K^{-1} \Phi_N \right| \le C |\Phi_N|
\]
for all $N \ge 2$.

\begin{theorem}
\label{thm:wienerWintner}
\label{thm:pointwise-ET-with-distal-weights}
Let $G$ be a countable discrete amenable group, let $\Phi$ be a tempered \Folner{} sequence on $G$ and let $\mathbf{Y} = (Y,S,\mu_Y)$ be an ergodic $G$ system.
For every $\phi$ in $\lp^1(Y,\mu_Y)$ there is a conull set $Y' \subset Y$ with the following property:
For any uniquely ergodic topological $G$ system $(X,T)$, with unique invariant measure $\mu_X$, such that the $G$ system $\mathbf{X}=(X,T,\mu_X)$ is measurably distal and Kronecker disjoint from $\mathbf{Y}$, and for any $f\in\cont(X)$, any $x\in X$ and any $y\in Y'$ we have
\begin{equation}
\label{eqn:wwGoal}
\lim_{N \to\infty} \frac{1}{|\Phi_N|} \sum_{g \in \Phi_N} f(T^g x)\phi(S^g y) = \int f \intd \mu_X \int \phi \intd \mu_Y.
\end{equation}
\end{theorem}
One can quickly derive the classical Wiener-Wintner theorem \cite{Wiener_Wintner41} from \cref{thm:pointwise-ET-with-distal-weights}, which we do in \cref{sec:wienerWintner}.

The second type of application that we offer is to the theory of multiple recurrence.
It is somewhat surprising that only Kronecker disjointness is needed for the following theorems, even though multicorrelations are typically governed by nilrotations (cf.\ \cite{Host_Kra05,MR2257397,Leibman10}), which in general are of a higher complexity than rotations on compact abelian groups.
\change{However, as nilsystems are distal, \cref{thm:bergWeylDisjoint} allows us to deduce disjointness of nilsystems from disjointness of their Kronecker factors.}

\begin{theorem}
\label{thm:product-system-orthogonality}
Let $\mathbf{X}=(X,T,\mu_X)$ and $\mathbf{Y}=(Y,S,\mu_Y)$ be ergodic $\Z$ systems and assume $\mathbf{X}$ and $\mathbf{Y}$ are Kronecker disjoint.
Then for every $k,\ell\in\N$, any $f_1,\dots,f_k\in \lp^\infty(X,\mu_X)$ and any $g_1,\dots,g_\ell\in \lp^\infty(Y,\mu_Y)$ we have
\[
\lim_{N\to\infty}
\frac{1}{N} \sum_{n=1}^N \prod_{i=1}^k \prod_{j=1}^\ell T^{in}f_i \; S^{jn}g_j=\left(\lim_{N\to\infty}\frac1N\sum_{n=1}^N\prod_{i=1}^k T^{in}f_i\right) \left( \lim_{N\to\infty}\frac1N\sum_{n=1}^N\prod_{j=1}^\ell S^{jn}g_j\right)
\]
in $\lp^2(X\times Y,\mu_X \otimes \mu_Y)$.
\end{theorem}

\begin{theorem}\label{thm:mrww-1}
Let $(Y,S)$ be a topological $\Z$ system and let $\mu_Y$ be an ergodic $S$ invariant Borel probability measure on $Y$.
For all $G\in \lp^1(Y,\mu_Y)$ there exists a set $Y'\subset Y$ with $\mu_Y(Y')=1$ such that for any ergodic $\Z$ system $(X,T,\mu_X)$ which is Kronecker disjoint from $(Y,S,\mu_Y)$, any $k\in\N$, any $f_1,\dots,f_k\in \lp^\infty(X,\mu_X)$ and any $y\in Y'$:
\[
\lim_{N \to \infty} \frac{1}{N} \sum_{n=1}^N G(S^ny)\prod_{i=1}^k T^{in} f_i
=
\left( \int_Y G \d\mu_Y \right) \cdot \left( \lim_{N\to\infty}\frac{1}{N} \sum_{n=1}^N \prod_{i=1}^k T^{in} f_i \right)
\]
in $\lp^2(X,\mu_X)$.
Moreover, if $(Y,S)$ is uniquely ergodic and $G\in\cont(Y)$ then we can take $Y'=Y$.
\end{theorem}

\paragraph{Structure of the paper:}
In \cref{sec:prelims} we review basic results and facts regarding Kronecker factors and distal systems, which are needed in the subsequent sections.
In \cref{sec:samequasidisjoint} we discuss Berg's notion of quasi-disjointness for $\Z$ systems in more detail and give a proof of \cref{thm_samequasidisjoint}.
In \cref{sec:distalSystems} we provide a proof of \cref{thm:main-result} by showing that quasi-disjointness lifts through group-extensions, is preserved by passing to factors and is preserved under taking inverse limits.
Sections \ref{sec:wienerWintner} and \ref{sec:multicorrelations} contain numerous applications of our main results to questions about pointwise convergence in ergodic theory and to the theory of multiple recurrence, including proofs of Theorems \ref{thm:product-system-orthogonality} and \ref{thm:mrww-1}.
Finally, in \cref{sec:open-Q} we formulate some natural open questions.

\paragraph{Acknowledgements:}
We would like to thank Vitaly Bergelson for helpful comments on an early version of the paper, Eli Glasner for fruitful discussions on disjointness of measurably distal systems, and Pavel Zorin-Kranich for useful remarks on Wiener--Wintner results.
\change{We would like to thank the referee for suggesting improvements to the readability of the paper.}
We would also like to thank the referee of an earlier version of this paper for pointing out errors in the statements of Theorems~\ref{thm:3.4second} and \ref{thm:3.4second-star} in light of work by Liang and Qiu, which have since been corrected.
The first author was supported by the NSF grant DMS-1700147. The third author gratefully acknowledges the support of the NSF via grants DMS-1246989 and DMS-1703597.

\section{Preliminaries}
\label{sec:prelims}

In this section we present various preliminary results, which will be of use throughout the paper, on $G$ systems, their Kronecker factors and their joinings.
We conclude with a brief discussion of topologically and measurably distal systems.
Throughout this paper $G$ denotes a countable group.

\subsection{Measure preserving systems}
\label{sec:mpa}
By a \define{topological $G$ system} we mean a pair $(X,T)$ where $X$ is a compact metric space and $T$ is a continuous left action of $G$ on $X$.
A \define{$G$ system} is a tuple $\mathbf{X} = (X,T,\mu_X)$ where $(X,T)$ is a topological $G$ system and $\mu_X$ is a $T$ invariant Borel probability measure on $X$.
The \define{product} of two $G$ systems $\mathbf{X} = (X,T,\mu_X)$ and $\mathbf{Z} = (Z,R,\mu_Z)$ is the system $\mathbf{X} \times \mathbf{Z} = (X \times Z, T \times R, \mu_X \otimes \mu_Z)$ where $T \times R$ is the diagonal action $(T \times R)^g = T^g \times R^g$.

A $G$ system $\mathbf{Z} = (Z,R,\mu_Z)$ is a \define{factor} of a $G$ system $\mathbf{X} = (X,T,\mu_X)$ if there is a $G$ invariant, conull (i.e. full measure) subset $X'$ of $X$ and a measurable, measure-preserving, $G$ equivariant map $X' \to Z$.
Any such map, together with its conull, invariant domain, is called a \define{factor map}.

Given a factor $\mathbf{Z}$ of a $G$ system $\mathbf{X}$ the associated factor map induces an isometric embedding of $\lp^2(Z,\mu_Z)$ in $\lp^2(X,\mu_X)$.
Denote by $\condex{\cdot}{\mathbf{Z}}$ the orthogonal projection from $\lp^2(X,\mu_X)$ to this embedded copy of $\lp^2(Z,\mu_Z)$.
Given $f \in \lp^2(X,\mu_X)$, we can think of $\condex{f}{\mathbf{Z}}$ as a function either on $X$ or on $Z$.

\subsection{Disintegrations}
\label{sec:disintegrations}

Given a factor map $\pi : \mathbf{X} \to \mathbf{Y}$ of $G$ systems one can always find an almost-surely defined, measurable family $y \mapsto \mu_y$ of Borel probability measures on $X$ such that
\[
\int f \intd \mu_X = \iint f \intd \mu_y \intd \mu_Y(y)
\]
for all $f$ in $\lp^1(X,\mu_X)$ and that $T^g \mu_y = \mu_{S^g y}$ for all $g \in G$ almost surely.
Moreover, the family $y \mapsto \mu_y$ is uniquely determined almost surely by these properties.
If $\mathbf{Y}$ is the factor corresponding (via \cite[Corollary~2.2]{Zimmer76a}) to the $\sigma$-algebra of $T$ invariant sets then the resulting disintegration $y \mapsto \mu_y$ is a version of the ergodic decomposition of $\mu_X$.
We refer the reader to \cite[Chapter~5]{MR2723325} for details on disintegrations of measures.

\subsection{The Kronecker factor}
\label{subsec:kroneckerFactor}

Every $G$ system $\mathbf{X}= (X,T,\mu_X)$ induces a right action of $G$ on $\lp^2(X,\mu_X)$ defined by $(T^g f)(x) = f(T^g x)$.
A function $f\in \lp^2(X,\mu_X)$ is \define{almost periodic} if its orbit $\{T^gf:g\in G\}$ has compact closure in the strong topology of $\lp^2(X,\mu_X)$.
We write $\ap(\mathbf{X})$ for the closed subspace of $\lp^2(X,\mu_X)$ spanned by almost periodic functions.
A $G$ system $\mathbf{X}$ is \define{almost periodic} if $\lp^2(X,\mu_X) = \ap(\mathbf{X})$.
It follows from \cite[Lemma~4.3]{Leeuw_Glicksberg61} that $\ap(\mathbf{X}) $ coincides with  the subspace of $\lp^2(X,\mu_X)$ spanned by finite-dimensional, $T$ invariant subspaces of $\lp^2(X,\mu_X)$.
There is a $G$ invariant, countably generated sub-$\sigma$-algebra $\mathscr{A}$ of the Borel $\sigma$-algebra of $X$ such that $\ap(\mathbf{X}) = \lp^2(X,\mathscr{A},\mu_X)$ (cf.\ \cite[Lemma 3.1]{MR1191743}).
By \cite[Corollary~2.2]{Zimmer76a} the $\sigma$-algebra $\mathscr{A}$ corresponds to a factor $\kron \mathbf{X}$ of $\mathbf{X}$ called the \define{Kronecker factor}.

\begin{proposition}
\label{prop:K-factor-of-product}
For any two $G$ systems $\mathbf{X}$ and $\mathbf{Z}$ we have $\ap( \mathbf{X} \times \mathbf{Z} ) = \ap(\mathbf{X}) \otimes \ap(\mathbf{Z})$.
\end{proposition}
\begin{proof}
Given a $G$ system $\mathbf{Y}$ denote by $\wm(\mathbf{Y})$ the closure in $\lp^2(Y,\mu_Y)$ of the collection of vectors $f$ with the property that $0$ belongs to the weak closure of $\{ T^g f : g \in G \}$.
By \cite[Corollary 4.12]{Leeuw_Glicksberg61} one can write $\lp^2(Y,\mu_Y)$ as the direct sum $\ap(\mathbf{Y}) \oplus \wm(\mathbf{Y})$.

Fix now $G$ systems $\mathbf{X} = (X,T,\mu_X)$ and $\mathbf{Z} = (Z,R,\mu_Z)$.
We have
\[
\lp^2(X \times Z, \mu_X \otimes \mu_Z))
=
(\ap(\mathbf{X}) \otimes \ap(\mathbf{Z}))
\oplus
(\wm(\mathbf{X}) \otimes \ap(\mathbf{Z}))
\oplus
(\ap(\mathbf{X}) \otimes \wm(\mathbf{Z}))
\oplus
(\wm(\mathbf{X}) \otimes \wm(\mathbf{Z}))
\]
holds.
Certainly $\ap(\mathbf{X}) \otimes \ap(\mathbf{Z}) \subset \ap(\mathbf{X} \times \mathbf{Z})$.
If $f$ belongs to $\wm(\mathbf{X})$ then $f \otimes g$ belongs to $\wm(\mathbf{X} \times \mathbf{Z})$ for every $g$ in $\lp^2(Z,\mu_Z)$.
Similarly, if $g$ belongs to $\wm(\mathbf{Z})$ then $f \otimes g$ belongs to $\wm(\mathbf{X} \times \mathbf{Z})$ for every $f$ in $\lp^2(Z,\mu_Z)$.
This gives
\[
(\wm(\mathbf{X}) \otimes \ap(\mathbf{Z}))
\oplus
(\ap(\mathbf{X}) \otimes \wm(\mathbf{Z}))
\oplus
(\wm(\mathbf{X}) \otimes \wm(\mathbf{Z}))
\subset
\wm(\mathbf{X} \otimes \mathbf{Z})
\]
and the result follows.
\end{proof}

When $\mathbf{X}$ is ergodic \cite[Theorem~1]{Mackey64} allows us to model the system $\kron \mathbf{X}$ as a homogeneous space $K/H$ where $K$ is a compact group and $H$ is a closed subgroup with the action of $G$ on $K/H$ given by a homomorphism $G \to K$ with dense image.
Write $\kron X$ for the underlying space of any such modeling of $\kron \mathbf{X}$.
For ergodic $\Z$ systems the Kronecker factor can be explicitly described by the system's discrete spectrum.

\begin{definition}
\label{def:discrete-spectrum}
Let $\mathbf{X}=(X,T,\mu_X)$ be a $\Z$ system. The \define{discrete spectrum} of $\mathbf{X}$, denoted by $\Eig(\mathbf{X})$, is defined to be the set of all eigenvalues of $T$ when viewed as a unitary operator $T:\lp^2(X,\mu_X) \to\lp^2(X,\mu_X)$,
$$
\Eig(\mathbf{X})=\{\zeta\in\C: \exists f\in\lp^2(X,\mu_X)~\text{with}~f\neq 0~\text{and}~ Tf=\zeta f\}.
$$
\end{definition}

Given two ergodic $G$ systems $\mathbf{X}$ and $\mathbf{Z}$ we can (using the Bohr compactification of $G$, for instance) assume that $\kron \mathbf{X}$ and $\kron \mathbf{Z}$ are modeled by homogeneous spaces of the same compact topological group $K$.
That is, we may assume $\kron X = K/H_\mathbf{X}$ and $\kron Z = K/H_\mathbf{Z}$ for some compact topological group $K$ and closed subgroups $H_\mathbf{X}$, $H_\mathbf{Z}$ thereof, the action of $G$ on both spaces determined by a homomorphism $G \to K$ with dense image.
The system $K/ \langle H_\mathbf{X},H_\mathbf{Z} \rangle$ is a factor of both $\mathbf{X}$ and $\mathbf{Z}$ and serves as a model for $\kron(\mathbf{X},\mathbf{Z})$ -- their \define{joint Kronecker factor} by which we mean the largest almost periodic factor of both $\mathbf{X}$ and $\mathbf{Z}$.

\begin{proposition}[cf.\ {\cite[Lemma~4.18]{Furstenberg81}}]
\label{prop:invariantVectors}
If $\mathbf{X}$ and $\mathbf{Z}$ are ergodic $G$ systems and $\mathbf{X} \times \mathbf{Z}$ is not ergodic then $\lp^2(X,\mu_X)$ and $\lp^2(Z,\mu_Z)$ contain isomorphic, finite-dimensional, invariant subspaces of non-constant functions.
\end{proposition}
\begin{proof}
If $\mathbf{X} \times \mathbf{Z}$ is not ergodic then there is a non-constant function $f$ in $\lp^2(X \times Z, \mu_X \otimes \mu_Z)$ that is $T \times R$ invariant.
Therefore $\bilin{f}{\phi} = \bilin{f}{(T \times R)^g \phi}$ for all $\phi\in\lp^2(X\times Z,\mu_X\otimes\mu_Z)$ and $g \in G$.
In particular, if $\phi \in \wm(\mathbf{X} \times \mathbf{Z})$ then $\bilin{f}{\phi} = 0$.
It follows that $f$ belongs to $\ap(\mathbf{X} \times \mathbf{Z})$ and therefore (by \cref{prop:K-factor-of-product}) to $\ap(\mathbf{X}) \otimes \ap(\mathbf{Z})$.
Now $\ap(\mathbf{X})$ and $\ap(\mathbf{Z})$ are spanned by finite-dimensional, invariant subspaces of $\lp^2(X,\mu_X)$ and $\lp^2(Z,\mu_Z)$ respectively so
\[
f = \sum_{\iota,\kappa} g_\iota \otimes h_\kappa
\]
where $\iota$ and $\kappa$ enumerate the finite dimensional subrepresentations of $\lp^2(X,\mu_X)$ and $\lp^2(Z,\mu_Z)$ respectively and $g_\iota \otimes h_\kappa$ is the projection of $f$ on the corresponding subrepresentation of $\lp^2(X \times Z, \mu_X \otimes \mu_Z)$.
The fact that $f$ is invariant implies only terms of the form $g_\iota \otimes h_{\iota^*}$ contribute to the sum, where $\iota^*$ denotes the contragradient of the representation $\iota$.
Since $f$ is non-constant there must be a non-trivial, finite-dimensional representation of $G$ that appears as a sub-representation of both $\lp^2(X,\mu_X)$ and $\lp^2(Z,\mu_Z)$.
\end{proof}

\subsection{Joinings}

Given $G$ systems $\mathbf{X} = (X,T,\mu_X)$ and $\mathbf{Z} = (Z,R,\mu_Z)$, a measure $\lambda$ on $X \times Z$ is a \define{joining} of $\mathbf{X}$ and $\mathbf{Z}$ if $\lambda$ is invariant under the diagonal action $T \times R$ and the pushforwards of $\lambda$ under the two coordinate projection maps $\pi_{\mathbf{X}}:X\times Z\to X$ and $\pi_{\mathbf{Z}} : X\times Z\to Z$ satisfy $\pi_{\mathbf{X}}(\lambda) = \mu_X$ and $\pi_\mathbf{Z}(\lambda)=\mu_Z$.
We  write $\joinings(\mathbf{X},\mathbf{Z})$ for the set of all joinings of $\mathbf{X}$ with $\mathbf{Z}$ and $\ergjoinings(\mathbf{X},\mathbf{Z})$ for the set of joinings of $\mathbf{X}$ with $\mathbf{Z}$ that are ergodic.
The product $\mu_X \otimes \mu_Z$ is always a joining of $\mathbf{X}$ and $\mathbf{Z}$.
When $\mathbf{X}$ and $\mathbf{Z}$ are ergodic the set $\ergjoinings(\mathbf{X},\mathbf{Z})$ is always non-empty because the measures in the ergodic decomposition of $\mu_X \otimes \mu_Z$ can be shown to be ergodic joinings of $\mathbf{X}$ and $\mathbf{Z}$.

One says that $G$ systems $\mathbf{X}$ and $\mathbf{Z}$ are \define{disjoint} if $\mu_X \otimes \mu_Z$ is their only joining.
We say that $G$ systems $\mathbf{X}$ and $\mathbf{Z}$ are \define{Kronecker disjoint} if their Kronecker factors $\kron \mathbf{X}$ and $\kron \mathbf{Z}$ are disjoint.

\subsection{Distal systems}
\label{subsec:prelim-distal}

Given a $G$ system $\mathbf{X} = (X,T,\mu_X)$ denote by $\Aut(\mathbf{X})$ the group of invertible, measurable, measure-preserving maps on $(X,\mu_X)$ that commute with $T$, where two automorphisms are identified if they coincide $\mu_X$ almost everywhere.
Since $X$ is a compact metric space, the group $\Aut(\mathbf{X})$ is metrizable and as such becomes a Polish topological group.
Given a compact subgroup $L$ of $\Aut(\mathbf{X})$, the associated sub-$\sigma$-algebra of $L$ invariant sets determines a factor of $\mathbf{X}$.
One says that a $G$ system $\mathbf{X}$ is a \define{group extension} of a $G$ system $\mathbf{Y}$ if $\mathbf{Y}$ is (isomorphic to) a factor of $\mathbf{X}$ via a compact subgroup of $\Aut(\mathbf{X})$ in the above fashion.
As stated in the introduction, a $G$ system is \define{measurably distal} if it belongs to the smallest class of $G$ systems that contains the trivial one-point system and is closed under group extensions, factor maps and inverse limits.

We recall that a topological $G$ system $(X,T)$ is \define{topologically distal} if
\[
\inf \{ \mathsf{d}(T^g x, T^g x') : g \in G \} > 0
\]
for all $x \ne x'$ in $X$.
In the case $G = \mathbb{Z}$ Parry~\cite{Parry68} modified this definition to apply to measure preserving actions.
We now recall Zimmer's generalization~\cite[Definition~8.5]{Zimmer76} of Parry's definition to actions of countable groups.
Given a $G$ system $\mathbf{X} = (X,T,\mu_X)$ a sequence $n \mapsto A_n$ of Borel subsets of $X$ with $\mu_X(A_n) > 0$ and $\mu_X(A_n) \to 0$ is called a \define{separating sieve} if there is a conull set $X' \subset X$ such that, whenever $x,x' \in X'$ and, for each $n \in \mathbb{N}$, one can find $g_n \in G$ with $\{ T^{g_n} x, T^{g_n} x' \} \subset A_n$, one has $x = x'$.
Zimmer~\cite[Theorem~8.7]{Zimmer76} proved that a non-atomic $G$ system $\mathbf{X}$ is measurably distal if and only if it has a separating sieve.
Using this characterization of distality one can show that every topological $G$ system $(X,T)$ that is topologically distal has the property that for every $T$ invariant Borel probability measure $\mu_X$ on $X$ the $G$ system $(X,T,\mu_X)$ is measurably distal.
Lindenstrauss~\cite{MR1709430} has proved a partial converse to this result for $\mathbb{Z}$ systems by showing that every measurably distal $\mathbb{Z}$ system can be modelled by a topologically distal $\mathbb{Z}$ system equipped with an invariant Borel probability measure.

\section{Proof of \cref{thm_samequasidisjoint}}
\label{sec:samequasidisjoint}

In this section we prove Theorem~\ref{thm_samequasidisjoint}.
It follows a preparatory discussion parameterizing the space of joinings of ergodic, almost-periodic $\mathbb{Z}$ systems and describing the ergodic decomposition of the product of two such $\mathbb{Z}$ systems.

As described in Subsection~\ref{subsec:kroneckerFactor} the Kronecker factor of an ergodic $\Z$ system can be modeled as an ergodic rotation on a compact abelian group.
Given two ergodic rotations on compact, abelian groups ${\mathbf X}=(X,T,\mu_X)$ and ${\mathbf Y}=(Y,S,\mu_Y)$, their ergodic joinings can be easily described as follows:
Let $e_X$ be the identity of the compact abelian group $X$, let $e_Y$ be the identity of $Y$ and let $H$ be the subgroup
\[
H = \overline{\{(T\times S)^n(e_X,e_Y):n\in\Z\}}
\]
of $X\times Y$.
Given an ergodic joining $\lambda\in\ergjoinings(\mathbf{X},\mathbf{Y})$, let $(x_0,y_0)\in X\times Y$ be a generic point.
The support of $\lambda$ is the orbit closure of $(x_0,y_0)$, which is $(x_0,y_0)+H$. Hence the pushforward of $\lambda$ under the map $(x,y)\mapsto(x-x_0,y-y_0)$ is a measure on $H$ invariant under $T\times S$, and hence must be the Haar measure on $H$.
It follows that $\lambda$ is the Haar measure on the coset $(x_0,y_0)+H$.

Now let $K = (X\times Y)/H$ be the group of cosets of $H$ and write $\mu_K$ for Haar measure on $K$.
Define a map $R : K \to K$ by $R:(x,y)+H\mapsto(Tx,y)+H=(x,S^{-1}y)+H$.
Let $\alpha:X\to K$ be defined by $\alpha(x)=(x,e_Y)+H$ and let $\beta:Y\to K$ be defined by $\beta(y)=(e_X,-y)+H$.
It's easy to check that both $\alpha$ and $\beta$ are factor maps onto $(K,R)$ and hence either ${\mathbf X}\times{\mathbf Y}$ is ergodic (and thus $K=\{\id\}$), or $\mathbf{X}$ and $\mathbf{Y}$ share a nontrivial common factor.


Notice that $\gamma(x,y)=(x,y)+H$ from ${\mathbf X}\times{\mathbf Y}$ to $K$ is the maximal invariant factor (because every invariant function is constant along cosets of $H$). In fact $(K,R)$ is the maximal common factor of $\mathbf{X}$ and $\mathbf{Y}$.

Since $\gamma^{-1}\big((x,y)+H\big)=(x,y)+H$, there exists exactly one ergodic joining living in that pre-image, namely, the Haar measure.
For each $k\in K$ let $\lambda_k\in\ergjoinings({\mathbf X},{\mathbf Y})$ be the unique joining such that $\lambda_k(\gamma^{-1}(k))=1$.
Observe that
\begin{equation}
\label{eq_kronergdecomposition}
\mu_X\otimes\mu_Y=\int_K\lambda_k \d\mu_K(k)
\end{equation}
is therefore the ergodic decomposition of $\mu_X \otimes \mu_Y$.

\begin{proof}[Proof of Theorem~\ref{thm_samequasidisjoint}]
First suppose that $\mathbf{X}$ and $\mathbf{Y}$ are quasi-disjoint.
Let $(K,\mu_K)$ be a model for their joint Kronecker factor and let $k \mapsto \lambda_k$ be a measurable map from $K$ into the space $\ergjoinings(\mathbf{X},\mathbf{Y})$ of ergodic joinings of $\mathbf{X}$ and $\mathbf{Y}$ such that $\lambda_k$ gives full measure to $\gamma^{-1}(k)$ for almost every $k$.
Then
\[
\eta = \int \lambda_k \intd \mu_K(k)
\]
is a joining of ${\mathbf X}$ and ${\mathbf Y}$.
We claim that the projection of $\eta$ to $\kron \mathbf{X} \times \kron \mathbf{Y}$ is the product measure.
Indeed, let $\pi:{\mathbf X}\times{\mathbf Y}\to\kron{\mathbf X}\times\kron{\mathbf Y}$ be the corresponding factor map.
We can decompose $\gamma=\tilde\gamma\circ\pi$ for some $\tilde\gamma:\kron X\times\kron Y\to K$.
Since $\pi(\lambda_k)$ gives full measure to $\tilde\gamma^{-1}(k)$, the discussion preceding this proof implies it is uniquely determined by $k$.
In particular, in view of \cref{eq_kronergdecomposition} we have
\[
\pi(\eta)=\int\pi(\lambda_k)\intd \mu_K(k)=\mu_{\kron{\mathbf X}}\otimes\mu_{\kron{\mathbf Y}}
\]
as claimed.
But then by quasi-disjointness $\eta = \mu_{\mathbf X} \otimes \mu_{\mathbf Y}$.
By uniqueness of the ergodic disintegration, the map $k\to\lambda_k$ is uniquely defined almost everywhere, so $\mathbf{X}$ and $\mathbf{Y}$ satisfy (BQD).

Conversely, suppose that $\mathbf{X}$ and $\mathbf{Y}$ satisfy (BQD).
Let $\eta$ be a joining of $\mathbf{X}$ and $\mathbf{Y}$ that projects to the product measure on $\kron \mathbf{X} \times \kron \mathbf{Y}$.
Then $\gamma \eta$ is the Haar measure on the maximal common Kronecker factor $\kron(\mathbf{X},\mathbf{Y})$ of $\mathbf{X}$ and $\mathbf{Y}$.
The ergodic disintegration of $\eta$ is the same as the disintegration of $\eta$ over $\kron(\mathbf{X},\mathbf{Y})$.
By (BQD), this disintegration is in turn the same as the disintegration of $\mu \otimes \nu$ over $\kron(\mathbf{X},\mathbf{Y})$.
Therefore $\eta = \mu \otimes \nu$.
\end{proof}

\section{Quasi-disjointness for measurably distal systems}
\label{sec:distalSystems}

In this section we give a proof of \cref{thm:main-result}. The proof is comprised of three parts, covered in the following three subsections. The first part consists of showing that quasi-disjointness lifts through group-extensions. The second part proves that quasi-disjointness is preserved when passing to a factor and the third part consists of showing that quasi-disjointness is preserved by inverse limits.
Since, starting from the trivial system, such operations exhaust the class of distal systems, these three parts combined indeed yield a complete proof of \cref{thm:main-result}.
We conclude this section with an example of a $\mathbb{Z}$ system that is not measurably distal but is quasi-disjoint from every ergodic system.

\subsection{Quasi-disjointness lifts through group extensions}

The purpose of this subsection is to prove the following result.

\begin{theorem}
\label{thm:groupextensions}
  Let ${\mathbf X}=(X,T,\mu_X)$, ${\mathbf Y}=(Y,S,\mu_Y)$ and ${\mathbf Z}=(Z,R,\mu_Z)$ be $G$ systems and assume that ${\mathbf X}$ is a group extension of ${\mathbf Y}$. If ${\mathbf Y}$ is quasi-disjoint from ${\mathbf Z}$, then so is ${\mathbf X}$.
\end{theorem}

For the proof of \cref{thm:groupextensions} we borrow ideas from the proof of \cite[Theorem 1.4]{Furstenberg67}.
See \cite[Theorem~3.30]{MR1958753} for a version of that result applicable to actions of countable groups.

\begin{proof}[Proof of \cref{thm:groupextensions}]
Let $\pi_{\mathbf{X},\kron\mathbf{X}}$ and $\pi_{\mathbf{Z},\kron\mathbf{Z}}$ denote the projection maps from $\mathbf{X}$ onto $\kron\mathbf{X}$ and from $\mathbf{Z}$ onto $\kron\mathbf{Z}$ respectively.
Let $\lambda\in\joinings(\mathbf{X},\mathbf{Z})$ be a joining of $\mathbf{X}$ and $\mathbf{Z}$ with the property that
$(\pi_{\mathbf{X},\kron\mathbf{X}}\times \pi_{\mathbf{Z},\kron\mathbf{Z}})(\lambda)=\mu_{\kron X}\otimes\mu_{\kron Z}$. We want to show that $\lambda=\mu_X\otimes \mu_Z$.

Let $\pi_{\mathbf{X},\mathbf{Y}}$ denote the factor map from $\mathbf{X}$ onto $\mathbf{Y}$. Note that ${(\pi_{\mathbf{X},\mathbf{Y}}\times\text{id}_{\mathbf{Z}})}(\lambda)$ is a joining of $\mathbf{Y}$ with $\mathbf{Z}$ whose projection onto the product of the Kronecker factors $\kron\mathbf{Y}\times\kron\mathbf{Z}$ equals $\mu_{KY}\otimes\mu_{KZ}$. Since ${\mathbf Y}$ is quasi-disjoint from ${\mathbf Z}$, we conclude that ${(\pi_{\mathbf{X},\mathbf{Y}}\times\text{id}_{\mathbf{Z}})}(\lambda)=\mu_Y\otimes\mu_Z$.

Since $\mathbf{X}$ is a group extension of $\mathbf{Y}$, there exists a compact group $L\leq \Aut(\mathbf{X})$ such that the factor $\mathbf{Y}$ corresponds to the sub-$\sigma$-algebra of $L$ invariant subsets of $\mathbf{X}$ (cf.\ Subsection~\ref{subsec:prelim-distal}).
Let $\mu_L$ denote the normalized Haar measure on $L$ and let $\mathbf{L}=(L,Q,\mu_L)$, where $Q$ is the action of $L$ on itself by left multiplication.
Given $\psi \in \lp^\infty(L,\mu_L)$ with $\psi\geq 0$ we define a new joining $\lambda_{\psi}\in\joinings(\mathbf{X},\mathbf{Z})$ by
\begin{equation}
\label{eqn:F-ac-measure}
\lambda_{\psi}(A) = \int_{X\times Z} \int_{K} \1_{A}(lx,z)\psi(l) \intd\mu_L(l)\intd\lambda(x,z)
\end{equation}
for all Borel sets $A \subset X \times Z$.
Let $\lambda_1$ denote the measure defined by \eqref{eqn:F-ac-measure} with $\psi = 1$.
We claim that $\lambda_1 = \mu_X\otimes\mu_Z$.
To verify this claim, let $f \in \lp^2(X,\mu_X)$, $g\in \lp^2(Z,\mu_Z)$ and define $f'(x)=\int_{L} f(lx)\intd\mu_L(l)$ for all $x\in X$.
Note that $f'$ is $L$ invariant and hence there exists $h\in \lp^2(Y,\mu_Y)$ such that $h\circ \pi_{\mathbf{X},\mathbf{Y}}=f'$.
We have
\begin{align*}
\int_{X\times Z} f\otimes g\intd\lambda_1
&
=
\int_{X\times Z} f'\otimes g \intd\lambda
\\
&
=
\int_{X\times Z} (h\otimes g)\circ (\pi_{\mathbf{X},\mathbf{Y}}\times\text{id}_{\mathbf{Z}}) \intd\lambda
\\
&
=
\int_{Y\times Z} h\otimes g\intd
{(\pi_{\mathbf{X},\mathbf{Y}}\times\text{id}_{\mathbf{Z}})}(\lambda)
\\
&
=
\int_{Y\times Z} h\otimes g \intd\mu_Y\otimes\mu_Z
\\
&
=
\int_{X\times Z} f'\otimes g \intd\mu_X\otimes\mu_Z
=
\int_{X\times Z} f\otimes g \intd\mu_X\otimes\mu_Z
\end{align*}
because $(\pi_{\mathbf{X},\mathbf{Y}} \times \id_\mathbf{Z}) \lambda = \mu_Y \otimes \mu_Z$.
This shows that indeed $\lambda_1 = \mu_X \otimes \mu_Z$.

Next, observe that
\begin{equation}
\label{eqn:radon-nikodym}
\lambda_{\psi}(A)
\le
\lambda_1(A) \norm{\psi}{\infty}
=
(\mu_X\otimes \mu_Z)(A) \norm{\psi}{\infty}.
\end{equation}
for all Borel sets $A \subset X \times Z$.
Inequality \eqref{eqn:radon-nikodym} shows that $\lambda_{\psi}$ is absolutely continuous with respect to $\mu_X\otimes \mu_Z$. Let $F_\psi$ denote the Radon-Nikodym derivative of $\lambda_{\psi}$ with respect to $\mu_X\otimes \mu_Z$.
It also follows form \eqref{eqn:radon-nikodym} that $\norm{F_\psi}{\infty} \leq \norm{\psi}{\infty}$ and so $F_\psi\in\lp^\infty(X \times Z, \mu_X \otimes \mu_Z)$.
Moreover, since $\lambda_\psi$ is a $T \times R$ invariant measure, we conclude that $F_\psi$ is a $T \times R$ invariant function in $\lp^\infty(X \times Z, \mu_X \otimes \mu_Z)$.
Since any $T \times R$ invariant function is almost periodic, it follows that $F_\psi\in\ap( \mathbf{X} \times \mathbf{Z} )$.
Therefore $F_\psi \in \ap(\mathbf{X}) \otimes \ap(\mathbf{Z})$ by \cref{prop:K-factor-of-product}.
It follows that for all $f\in \lp^2(X,\mu_X)$ and $g\in \lp^2(Z,\mu_Z)$,
\begin{align*}
\int_{X\times Z} f\otimes g\intd\lambda_{\psi}
&
=
\int_{X\times Z} F_\psi \cdot (f\otimes g)\intd(\mu_X\otimes\mu_Z)
\\
&
=
\int_{X\times Z} F_\psi \cdot \big(\E(f|\kron\mathbf{X})\otimes\E(g|\kron\mathbf{Z})\big)\intd(\mu_X\otimes\mu_Z)
\\
&
=
\int_{X\times Z} \E(f|\kron\mathbf{X})\otimes\E(g|\kron\mathbf{Z}) \intd\lambda_{\psi}
\\
&
=
\int_{X\times Z} \int_{L} \E(f|\kron\mathbf{X})(lx) \, \E(g|\kron\mathbf{Z})(z) \, \psi(l) \intd\mu_L(l)\intd\lambda(x,z)
\\
&
=
\int_{L}
\left(
\int_{X\times Z} \E(lf | \kron\mathbf{X})(x) \, \E(g|\kron\mathbf{Z})(z) \intd\lambda(x,z)\right) \psi(l) \intd\mu_L(l).
\end{align*}
Since $(\pi_{\mathbf{X},\kron\mathbf{X}}\times \pi_{\mathbf{Z},\kron\mathbf{Z}})(\lambda)=\mu_{\kron X}\otimes\mu_{\kron Z}$ we have
\begin{align*}
\int_{X\times Z}  \E(l f | \kron\mathbf{X}) \otimes \E(g|\kron\mathbf{Z})\intd\lambda
&
=
\int_{X\times Z}  \E(l f | \kron\mathbf{X}) \otimes \E(g|\kron\mathbf{Z})\intd(\mu_X\otimes\mu_Z)
\\
&
=
\int_{X\times Z}  f\otimes g \intd(\mu_X\otimes\mu_Z).
\end{align*}
for all $l \in L$.
We conclude that
\[
\int_{X\times Z} f\otimes g\intd\lambda_{\psi}
=
\left(\int_{X\times Z}  f\otimes g \intd(\mu_X\otimes\mu_Z)\right)\left( \int_{L}\psi(l)\intd\mu_L(l)\right)
\]
so in particular, for any $\psi\in\lp^\infty(L,\mu_L)$ with $\psi \geq 0$ and $\int_L \psi\intd\mu_L = 1$ the measure $\lambda_\psi$ coincides with $\mu_X\otimes\mu_Z$.

Finally, allowing $\psi$ to run though an approximate identity, one can approximate $\lambda$ by $\lambda_\psi$ and thereby conclude that $\lambda = \mu_X\otimes\mu_Z$, which finishes the proof.
\end{proof}

\subsection{Quasi-disjointness passes to factors}

In this subsection we prove the following theorem.

\begin{theorem}
\label{thm:factors}
Let $\mathbf{X}=(X,T,\mu_X)$, $\mathbf{Y}=(Y,S,\mu_Y)$ and $\mathbf{Z}=(Z,R,\mu_Z)$ be $G$ systems and suppose that $\mathbf{Y}$ is a factor of $\mathbf{X}$.
If $\mathbf{X}$ and $\mathbf{Z}$ are quasi-disjoint, then $\mathbf{Y}$ and $\mathbf{Z}$ are also quasi-disjoint.
\end{theorem}

For the proof of \cref{thm:factors} we need to recall the definition of relatively independent joinings.
Let $\mathbf{X}=(X,T,\mu_X)$, $\mathbf{Y}=(Y,S,\mu_Y)$ and $\mathbf{Z}=(Z,R,\mu_Z)$ be $G$ systems and suppose that $\mathbf{Y}$ is a factor of both $\mathbf{X}$ and $\mathbf{Z}$.
Let $\pi_{\mathbf{X},\mathbf{Y}}:X\to Y$ and $\pi_{\mathbf{Z},\mathbf{Y}}:Z\to Y$ denote the respective factor maps.
Using $\pi_{\mathbf{X},\mathbf{Y}}$ we can embed $\lp^2(Y,\mu_Y)$ into $\lp^2(X,\mu_X)$. Likewise, through $\pi_{\mathbf{Z},\mathbf{Y}}$ we can identify $\lp^2(Y,\mu_Y)$ as a subspace of $\lp^2(Z,\mu_Z)$.
The \define{relatively independent joining of $\mathbf{X}$ with $\mathbf{Z}$ over $\mathbf{Y}$} is the triple $\mathbf{X}\times_{\mathbf{Y}}\mathbf{Z}=(X\times Z,T\times R,\mu_{X}\otimes_{\mathbf{Y}}\mu_Z)$, where $\mu_{X}\otimes_{\mathbf{Y}}\mu_Z$ denotes the unique measure on $X\times Z$ with the property that
\begin{equation}
\label{eqn:rij-1}
\int_{X\times Z} f \otimes g \intd(\mu_{X}\otimes_{\mathbf{Y}}\mu_Z) = \int_{Y}\E(f|\mathbf{Y})\E(g|\mathbf{Y})\intd\mu_Y
\end{equation}
for all $f \in \lp^2(X,\mu_X)$ and all $g \in \lp^2(Z,\mu_Z)$.

\begin{lemma}
\label{lem:rij}
Let $\mathbf{X}=(X,T,\mu_X)$ and $\mathbf{Y}=(Y,S,\mu_Y)$ be $G$ systems and suppose that $\mathbf{Y}$ is a factor of $\mathbf{X}$. Then the relatively independent joining  $\kron\mathbf{X}\times_{\kron\mathbf{Y}}\mathbf{Y}$ is a factor of $\mathbf{X}$.
\end{lemma}
\begin{proof}
Let $\pi_{\mathbf{X},\mathbf{Y}}:X\to Y$ denote the factor map from $\mathbf{X}$ to $\mathbf{Y}$ and let $\pi_{\mathbf{X},\kron\mathbf{X}}:X\to \kron X$ denote the factor map from $\mathbf{X}$ to $\kron\mathbf{X}$.
Let $\tau:X\to  \kron X\times Y$ be defined as $\tau(x)=(\pi_{\mathbf{X},\kron\mathbf{X}}(x),\pi_{\mathbf{X},\mathbf{Y}}(x))$ for all $x\in X$.
We claim that $\tau$ is a factor map from $\mathbf{X}$ onto $\kron\mathbf{X}\times_{\kron\mathbf{Y}}\mathbf{Y}$.
Once this claim is verified, the proof is completed.

To show that $\tau$ is a factor map from $\mathbf{X}$ onto $\kron\mathbf{X}\times_{\kron\mathbf{Y}}\mathbf{Y}$, we must show that the pushforward of $\mu_X$ under $\tau$ equals $\mu_{\kron X}\otimes_{\kron\mathbf{Y}}\mu_Y$.
It suffices to show that
\begin{equation}
\label{eqn:rijif-1}
\int_{\kron X\times Y} f\otimes g \intd (\tau \mu_X)
=
\int_{\kron X\times Y} f\otimes g \intd \mu_{\kron X}\otimes_{\kron\mathbf{Y}}\mu_Y
\end{equation}
for all $f\in \lp^2(\kron X, \mu_K)$ and all $g\in \lp^2(Y,\mu_Y)$.


By definition, the right hand side of \eqref{eqn:rijif-1} equals $\int_{\kron Y}\E(f|\kron\mathbf{Y})\E(g|\kron\mathbf{Y})\intd\mu_{\kron Y}$, which can be rewritten as
\[
\int_{X}\E(f\circ \pi_{\mathbf{X},\kron\mathbf{X}} |\kron\mathbf{Y})\E(g\circ \pi_{\mathbf{X},\mathbf{Y}}|\kron\mathbf{Y})\intd\mu_{X}.
\]
The left hand side of \eqref{eqn:rijif-1} equals
\[
\int_{X}(f\circ \pi_{\mathbf{X},\kron\mathbf{X}})(g\circ \pi_{\mathbf{X},\mathbf{Y}})\intd\mu_X.
\]
Since $f\circ \pi_{\mathbf{X},\kron\mathbf{X}}=\E(f\circ \pi_{\mathbf{X},\kron\mathbf{X}}|\kron\mathbf{X})$ and $g\circ \pi_{\mathbf{X},\mathbf{Y}}=\E(g\circ \pi_{\mathbf{X},\mathbf{Y}}|\mathbf{Y})$
we have that
\[
\int_{X}(f\circ \pi_{\mathbf{X},\kron\mathbf{X}})(g\circ \pi_{\mathbf{X},\mathbf{Y}})\intd\mu_X =
\int_{X}\E(f\circ \pi_{\mathbf{X},\kron\mathbf{X}}|\kron\mathbf{X})\E(g\circ \pi_{\mathbf{X},\mathbf{Y}}|\mathbf{Y}) \intd\mu_X.
\]
Hence \eqref{eqn:rijif-1} is equivalent to
\begin{equation}
\label{eqn:rijif-2}
\int_{X}\E(f\circ \pi_{\mathbf{X},\kron\mathbf{X}}|\kron\mathbf{X})\E(g\circ \pi_{\mathbf{X},\mathbf{Y}}|\mathbf{Y}) \intd\mu_X
=
\int_{X}\E(f\circ \pi_{\mathbf{X},\kron\mathbf{X}} |\kron\mathbf{Y})\E(g\circ \pi_{\mathbf{X},\mathbf{Y}}|\kron\mathbf{Y})\intd\mu_{X}.
\end{equation}
However, since $\E(\cdot|\kron\mathbf{X})$ and $\E(\cdot|\mathbf{Y})$ are orthogonal projections onto $\ap({\mathbf{X}})$ and $\lp^2(Y,\mu_Y)$ respectively, it follows immediately from $\ap(\mathbf{X})\cap \lp^2(Y,\mu_Y)=\ap(\mathbf{Y})$ that \eqref{eqn:rijif-2} is true.
\end{proof}

\begin{lemma}
\label{lem:lift-of-joining}
Let $\mathbf{A} = (A,T_A,\mu_A)$, $\mathbf{B} = (B,T_B,\mu_B)$, $\mathbf{C} = (C,T_C,\mu_C)$ and $\mathbf{D} = (D,T_D,\mu_D)$ be $G$ systems.
If $\mathbf{B}$ is a factor of $\mathbf{A}$ and $\mathbf{D}$ is a factor of $\mathbf{C}$ then the induced map from $\joinings(\mathbf{A}, \mathbf{C})$ to $\joinings(\mathbf{B},\mathbf{D})$ is surjective.
\end{lemma}
\begin{proof}
Write $\pi_{\mathbf{A},\mathbf{B}}$ and $\pi_{\mathbf{C},\mathbf{D}}$ for the factor maps.
Let $b \mapsto \mu_{A,b}$ and $d \mapsto \mu_{C,d}$ be disintegrations (cf.\ Subsection~\ref{sec:disintegrations}) of $\mu_A$ and $\mu_C$ over $\mathbf{B}$ and $\mathbf{D}$ respectively.
We have $T_A^g \mu_{A,b} = \mu_{A,T_B^g b}$ and $T_C^g \mu_{C,d} = \mu_{C,T_D^g d}$ almost surely for all $g \in G$.

Fix a joining $\lambda_{B,D}$ of $\mathbf{B}$ and $\mathbf{D}$.
We claim that
\[
\lambda_{A,C} = \int \mu_{A,b} \otimes \mu_{C,d} \intd \lambda_{B,D}(b,d)
\]
is a joining of $\mathbf{A}$ and $\mathbf{C}$ with $(\pi_{\mathbf{A},\mathbf{B}} \times \pi_{\mathbf{C},\mathbf{D}}) \lambda_{A,C} = \lambda_{B,D}$.
First note that
\[
\lambda_{A , C}(E \times C) = \int \mu_{A,b}(E) \intd \lambda_{B,D}(b,d) = \int \mu_{A,b}(E) \intd \mu_B(b) = \mu_A(E)
\]
for all Borel sets $E \subset A$ because $\lambda_{B,D}$ is a joining, so the left marginal of $\lambda_{A,C}$ is $\mu_A$.
Similarly, its right marginal is $\mu_C$.

For all $f \in \cont(B)$ and all $h \in \cont(D)$ we have
\[
\iint f \otimes h \intd(\mu_{A,b} \otimes \mu_{C,d}) \intd \lambda_{B,D}(b,d)
=
\int \int f \intd \mu_{A,b} \int h \intd \mu_{C,d} \intd \lambda_{B,D}(b,d)
=
\int f \otimes h \intd \lambda_{B,D}
\]
by disintegration properties, so $(\pi_{\mathbf{A},\mathbf{B}} \times \pi_{\mathbf{C},\mathbf{D}})\lambda_{A,C} = \lambda_{B,D}$.

Finally, for any $f \in \cont(A)$ and any $h \in \cont(C)$ we calculate that
\begin{align*}
\iint T_A^g f \otimes T_C^g h \intd(\mu_{A,b} \otimes \mu_{C,d}) \intd \lambda_{B,D}(b,d)
&
=
\iint f \otimes h \intd(\mu_{A,T_B^g b} \otimes \mu_{C,T_D^g d}) \intd \lambda_{B,D}(b,d)
\\
&
=
\iint f \otimes h \intd(\mu_{A,b} \otimes \mu_{C,d}) \intd \lambda_{B,D}(b,d)
\\
&
=
\int f \otimes h \intd \lambda_{A,C}
\end{align*}
so $\lambda_{A,C}$ is $T_A \times T_C$ invariant.
\end{proof}

From \cref{lem:lift-of-joining} we obtain the following immediate corollary.

\begin{corollary}
\label{cor:lift-of-joining}
Let $\mathbf{A}=(A,T_A,\mu_A)$, $\mathbf{B}=(B,T_B,\mu_B)$ and $\mathbf{C}=(C,T_C,\mu_C)$ be $G$ systems and suppose that $\mathbf{B}$ is a factor of $\mathbf{A}$.
Let $\pi_{\mathbf{A},\mathbf{B}}$ denote the factor map from $\mathbf{A}$ onto $\mathbf{B}$.
Then for any joining $\lambda\in\joinings(\mathbf{B},\mathbf{C})$ there exists a joining $\lambda'\in\joinings(\mathbf{A},\mathbf{C})$ such that the pushforward of $\lambda'$ under the factor map $\pi_{\mathbf{A},\mathbf{B}}\times \id_{\mathbf{C}}$ equals $\lambda$.
\[\begin{tikzcd}
\mathbf{A}
\arrow[swap]{d}{\pi_{\mathbf{A},\mathbf{B}}}
&
\mathbf{A}\times\mathbf{C}
\arrow{d}{\pi_{\mathbf{A},\mathbf{B}}\times \id_{\mathbf{C}}}
\\
\mathbf{B}
&
\mathbf{B}\times\mathbf{C}
\end{tikzcd}
\]
\end{corollary}

We also need the following lemmas for the proof of Theorem~\ref{thm:factors}.

\begin{lemma}
\label{lema:lift-of-ap}
Let $\pi : \mathbf{X} \to \mathbf{Y}$ be a factor map of $G$ systems.
If $f \in \ap(\mathbf{Y})$ then $f \circ \pi \in \ap(\mathbf{X})$.
\end{lemma}
\begin{proof}
This follows from the fact that $\pi : \lp^2(Y,\mu_Y) \to \lp^2(X,\mu_X)$ is an isometric embedding.
\end{proof}

\begin{lemma}
\label{lema:lift-of-wm}
Let $\pi : \mathbf{X} \to \mathbf{Y}$ be a factor map of $G$ systems.
Let $f \in \wm(\mathbf{Y})$ and define $h=f \circ \pi$.
Then $h \in \wm(\mathbf{X})$.
\end{lemma}
\begin{proof}
We need to prove that 0 belongs to the weak closure of $\{ T^g h : g \in G \}$ in $\lp^2(X,\mu_X)$.
Notice that $\condex{T^g h}{\mathbf{Y}}=T^g\condex{h}{\mathbf{Y}}=T^gh$.
Fix $\xi$ in $\lp^2(X,\mu_X)$.
We have
\[
\bilin{T^g h}{\xi}
=
\bilin{\condex{T^g h}{\mathbf{Y}}}{\condex{\xi}{\mathbf{Y}}}
\]
so the fact that $f \in \wm(\mathbf{Y})$ implies that $\{ T^g h : g \in G \}$ contains 0 in its closure.
\end{proof}

\begin{proof}[Proof of \cref{thm:factors}]
Let $\pi_{\mathbf{Y},\kron\mathbf{Y}}$ and $\pi_{\mathbf{Z},\kron\mathbf{Z}}$ denote the factor maps from $\mathbf{Y}$ onto $\kron\mathbf{Y}$ and from $\mathbf{Z}$ onto $\kron\mathbf{Z}$ respectively.
Let $\lambda \in \joinings(\mathbf{Y},\mathbf{Z})$ be a joining of $\mathbf{Y}$ and $\mathbf{Z}$ with the property that $(\pi_{\mathbf{Y},\kron\mathbf{Y}}\times \pi_{\mathbf{Z},\kron\mathbf{Z}})(\lambda)=\mu_{\kron Y} \otimes\mu_{\kron Z}$.
Let $\mathbf{W}=(Y\times Z, S\times R, \lambda)$.
Our goal is to show that $\mathbf{W}=\mathbf{Y}\times\mathbf{Z}$, or equivalently, that $\lambda=\mu_Y\otimes \mu_Z$.

Observe that $\kron\mathbf{Y}$ is a factor of both $\mathbf{W}$ and $\kron\mathbf{X}$.
Hence we can consider the relatively independent joining $\kron\mathbf{X}\times_{\kron\mathbf{Y}}\mathbf{W}$ with corresponding measure $\mu_{\kron X}\otimes_{\kron\mathbf{Y}}\lambda$.
Note that the underlying space of $\kron\mathbf{X}\times_{\kron\mathbf{Y}}\mathbf{W}$ is  $\kron X \times Y \times Z$.
Let $\pi_{1,2} :\kron X \times Y \times Z \to \kron X \times Y$ denote the projection onto the first and second coordinates and let $\pi_3 : \kron X \times Y \times Z\to Z$ denote the projection onto the third coordinate.
Observe that $\pi_3$ is a factor map from $\kron\mathbf{X} \times_{\kron\mathbf{Y}}\mathbf{W}$ onto $\mathbf{Z}$ and $\pi_{1,2}$ is a factor map from $\kron\mathbf{X}\times_{\kron\mathbf{Y}}\mathbf{W}$ onto $\kron\mathbf{X}\times_{\kron\mathbf{Y}}\mathbf{Y}$, the relatively independent joining of $\kron\mathbf{X}$ with $\mathbf{Y}$ over $\kron\mathbf{Y}$.
This shows that $\mu_{\kron X} \otimes_{\kron\mathbf{Y}} \lambda \in \joinings(\kron \mathbf{X} \times_{\kron\mathbf{Y}} \mathbf{Y}, \mathbf{Z})$.

Let $\tau$ denote the factor map from $\mathbf{X}$ to $\kron\mathbf{X}\times_{\kron\mathbf{Y}}\mathbf{Y}$ in the proof of \cref{lem:rij}.
We can now apply \cref{cor:lift-of-joining} with $\mathbf{A}=\mathbf{X}$, $\mathbf{B}=\kron\mathbf{X}\times_{\kron\mathbf{Y}}\mathbf{Y}$ and $\mathbf{C}=\mathbf{Z}$ to find a joining $\rho\in \joinings(\mathbf{X},\mathbf{Z})$ with the property that $(\tau \times \id_{\mathbf{Z}})(\rho) = \mu_{\kron X}\otimes_{\kron\mathbf{Y}}\lambda$.
Let $\WW=(X\times Z,T\times R,\rho)$.

Let $\pi_{1,3} : \kron X \times Y \times Z \to \kron X \times Z$ denote the projection onto the first and third coordinates.
We claim that $\pi_{1,3}(\mu_{\kron X} \otimes_{\kron \mathbf{Y}} \lambda) = \mu_{\kron X} \otimes \mu_Z$.
This claim implies that the diagram
\[
\begin{tikzcd}
\WW
\arrow{rr}{\tau\times\id_{\mathbf{Z}}}
\arrow[swap]{rrrrdd}{\pi_{\mathbf{X},\kron\mathbf{X}}\times \pi_{\mathbf{Z},\kron\mathbf{Z}}}
&
&
\kron\mathbf{X}\times_{\kron\mathbf{Y}}\mathbf{W}
\arrow{rr}{\pionethree}
&
&
\kron\mathbf{X}\times \mathbf{Z}
\arrow{dd}{\id_{\kron\mathbf{X}}\times\pi_{\mathbf{Z},\kron\mathbf{Z}}}
\\
\\
&
&
&
&
\kron\mathbf{X}\times\kron\mathbf{Z}
\end{tikzcd}
\]
of factor maps commutes, giving $(\pi_{\mathbf{X},\kron \mathbf{X}} \times \pi_{\mathbf{Z},\kron \mathbf{Z}}) \rho = \mu_{\kron X} \otimes \mu_{\kron Z}$ whence $\rho = \mu_X \otimes \mu_Z$ because $\mathbf{X}$ and $\mathbf{Z}$ are assumed to be quasi-disjointness.
If follows that the measure $(\tau \times \id_{\mathbf{Z}})(\rho)$ on $\kron\mathbf{X}\times_{\kron\mathbf{Y}}\mathbf{W}$ is the product of the measures from $\kron\mathbf{X}\times_{\kron\mathbf{Y}}\mathbf{Y}$ and $\mathbf{Z}$ and hence $\lambda = \mu_Y \otimes \mu_Z$ as desired.

It remains to prove the claim.
Fix $f$ in $\ap(\mathbf{X})$ and $\phi$ in $\lp^2(Z,\mu_Z)$.
We have
\[
\int f \otimes \phi \intd \pi_{1,3} (\mu_{\kron X} \otimes_{\kron \mathbf{Y}} \lambda)
=
\int f \otimes 1 \otimes \phi \intd (\mu_{\kron X} \otimes_{\kron \mathbf{Y}} \lambda)
=
\int \condex{f}{\kron\mathbf{Y}} \condex{1 \otimes\phi}{\kron\mathbf{Y}} \intd \mu_{\kron Y}
\]
so it suffices to prove that
\begin{equation}
\label{eqn:condexGoal}
\condex{1 \otimes \phi}{\kron \mathbf{Y}} = \int \phi \d \mu_Z
\end{equation}
in $\lp^2(Y \times Z, \lambda)$.
Fix $\psi$ in $\ap(\mathbf{Y})$.
We have
\[
\int (\psi \otimes 1) (1 \otimes \phi) \d \lambda
=
\int (\psi \otimes 1) \condex{1 \otimes \phi}{\kron \mathbf{W}} \d \lambda
\]
because $\psi \otimes 1$ is in $\ap(\mathbf{W})$ by \cref{lema:lift-of-ap}.
But
\[
\condex{1 \otimes \phi}{\kron \mathbf{W}}
=
\condex{1 \otimes \condex{\phi}{\kron \mathbf{Z}}}{\kron \mathbf{W}}
\]
by Lemmas~\ref{lema:lift-of-ap} and \ref{lema:lift-of-wm} upon writing $\phi = \condex{\phi}{\kron \mathbf{Z}} + ( \phi - \condex{\phi}{\kron \mathbf{Z}})$.
So we calculate that
\begin{align*}
\int (\psi \otimes 1) (1 \otimes \phi) \d \lambda
&
=
\int (\psi \otimes 1) (1 \otimes \condex{\phi}{\kron \mathbf{Z}}) \d \lambda
\\
&
=
\int \psi \otimes \condex{\phi}{\kron \mathbf{Z}} \d \lambda
\\
&
=
\int \condex{\psi}{\kron \mathbf{Y}} \otimes \condex{\phi}{\kron \mathbf{Z}} \d \lambda
=
\int \psi \d \mu_Y \int \phi \d \mu_Z
\end{align*}
where we have used, in the last equality, the fact that $\lambda$ projects to the product joining of $\kron \mathbf{Y}$ and $\kron \mathbf{Z}$.
This establishes \eqref{eqn:condexGoal} and therefore the claim.
\end{proof}

\subsection{Quasi-disjointness is preserved by inverse limits}

\begin{theorem}
Let ${\mathbf X}$ and ${\mathbf Z}$ be $G$ systems and assume that ${\mathbf X}$ is the inverse limit of a sequence $n \mapsto {\mathbf X}_n$ of $G$ systems.
If each ${\mathbf X}_n$ is quasi-disjoint from ${\mathbf Z}$, then so is ${\mathbf X}$.
\end{theorem}
\begin{proof}
Fix a joining $\lambda$ of $\mathbf{X}$ and $\mathbf{Z}$ whose projection to a joining of $\kron \mathbf{X}$ with $\kron \mathbf{Z}$ is the product measure.
For every $n$ the system $\kron \mathbf{X}_n$ is a factor of $\kron \mathbf{X}$ so $\lambda$ projects to the product joining of $\kron \mathbf{X}_n$ with $\kron \mathbf{Z}$.
As $\mathbf{X}_n$ and $\mathbf{Z}$ are assumed quasi-disjoint the projection of $\lambda$ to a joining of $\mathbf{X}_n$ with $\mathbf{Z}$ is the product measure $\mu_{X_n} \otimes \mu_Z$.
Therefore $\lambda(A \times B) = \mu_X(A) \mu_Z(B)$ for all measurable $A \subset X_n$ and all $B \subset Z$ for all $n \in \mathbb{N}$.
In other words $\lambda$ and $\mu_X \otimes \mu_Z$ agree on all measurable sets of the form $A \times B$ where $A \subset X_n$ for some $n \in \mathbb{N}$ and $B \subset Z$.
Since the $\sigma$-algebra generated by such sets is the Borel $\sigma$-algebra on $X \times Z$ we must have $\lambda = \mu_X \otimes \mu_Z$ as desired.
\end{proof}

\subsection{An example}

It follows from \cref{thm:bergWeylDisjoint} that any ergodic measurably distal $\Z$ system is quasi-disjoint from itself.
However, the converse is not true.

\begin{example}
There exists an ergodic $\mathbb{Z}$ system $\mathbf{X}$ which is not measurably distal but is quasi-disjoint from itself and every other ergodic system.
\end{example}
\begin{proof}
In \cite[Theorem 2.2]{Glasner_Weiss89} Glasner and Weiss construct a continuous $\mathbb{Z}$ action on a compact, metric space $X$ that is minimal and uniquely ergodic, but for which the corresponding $\mathbb{Z}$ system $(X,T,\mu_X)$ is not measurably distal.
Moreover, they prove for their system that the only invariant measure on the set $\tilde X=\{\mu\in{\mathcal M}(X):\pi_{{\mathbf X},\kron{\mathbf X}}(\mu)=\mu_{\kron{\mathbf X}}\}$ is the Dirac measure $\delta_{\mu_X}$ at the unique invariant measure $\mu_X\in\tilde X$ of $X$.

Given any ergodic system ${\mathbf Y}$, let $\lambda\in\ergjoinings({\mathbf X},{\mathbf Y})$ and assume that $\lambda$ projects to the product measure in $\kron{\mathbf X}\times\kron{\mathbf Y}$.
Let $\lambda=\int_Y\lambda_y\d\mu_Y(y)$ be the disintegration of $\lambda$ with respect to the factor map $(X\times Y,\lambda)\to(Y,\mu_Y)$.
Observe that
$$\mu_{\kron{\mathbf X}}\otimes\mu_{\kron{\mathbf Y}}=\int_Y\pi_{{\mathbf X},\kron{\mathbf X}}(\lambda_y)\otimes\pi_{{\mathbf Y},\kron{\mathbf Y}}(\delta_y)\d\mu_Y(y)$$ and so
$\pi_{{\mathbf X},\kron{\mathbf X}}(\lambda_y)=\mu_{\kron{\mathbf X}}$, which implies that $\lambda_y\in\tilde X$.

Let $\nu$ be the measure on $\tilde X$ obtained as the pushforward of $\mu_Y$ by the map $y\mapsto\lambda_y$. Since $\nu$ is invariant, it follows that $\nu=\delta_{\mu_X}$, and therefore $\lambda=\mu_X\otimes\mu_Y$.
We conclude that ${\mathbf X}$ and ${\mathbf Y}$ are quasi-disjoint as desired.
\end{proof}

\section{Proof of Theorem~\ref{thm:wienerWintner}}
\label{sec:wienerWintner}

In this section we prove \cref{thm:wienerWintner}.

\begin{proof}
[Proof of \cref{thm:wienerWintner}]
In view of a version for amenable groups of the Jewett-Krieger theorem~\cite{Rosenthal86}, we can assume that $\mathbf{Y}$ is uniquely ergodic.
Fix $(X,T,\mu_X)$ Kronecker disjoint from $(Y,S,\mu_Y)$, a function $f\in\cont(X)$ and a point $x\in X$.
Kronecker disjointness together with Theorem~\ref{thm:main-result} implies that the sequence
\[
\frac{1}{|\Phi_N|} \sum_{g \in \Phi_N} \delta_{(T^g x,S^g y)}
\]
of measures converges to $\mu_X \otimes \mu_Y$ for every $y\in Y$.
This implies that we can take $Y' = Y$ when $\phi$ is continuous.

For the general case, fix $\phi$ in $\lp^1(Y,\mu_Y)$.
Let $k \mapsto \phi_k$ be a sequence in $\cont(Y)$ with $\phi_k \to \phi$ in $\lp^1(Y,\mu_Y)$.
Let $Y'$ be the set of points $y\in Y$ such that
\begin{equation}
\label{eq_WW.0}
\lim_{N\to\infty} \frac{1}{|\Phi_N|} \sum_{g \in \Phi_N} \big|\phi_k(S^gy)-\phi(S^gy)\big|
=
\norm{\phi_k-\phi}{1}
\end{equation}
for all $k \in \mathbb{N}$.
In view of Lindenstrauss' pointwise ergodic theorem~\cite{MR1865397}, we have that $\mu_Y(Y')=1$.
Next, let $y\in Y'$, let $f\in \cont(X)$, $x\in X$ and let $k,N\in\N$.
By rescaling, assume that $\sup \{ f(x) : x \in X \} \leq 1$.
We have
\begin{gather}
\label{eq_WW.1}
\left| \frac{1}{|\Phi_N|} \sum_{g \in \Phi_N} f(T^g x)\phi_k(S^g y) - \frac{1}{|\Phi_N|} \sum_{g \in \Phi_N} f(T^g x)\phi(S^g y) \right| \leq
\frac1{|\Phi_N|}\sum_{g\in \Phi_N}\big|\phi_k(S^gy)-\phi(S^gy)\big|
\\
\label{eq_WW.2}
\lim_{N\to\infty}\frac{1}{|\Phi_N|} \sum_{g \in \Phi_N} f(T^g x)\phi_k(S^g y) = \int_Xf\d\mu_X\int_Y\phi_k\d\mu_Y
\end{gather}
and putting \eqref{eq_WW.0}, \eqref{eq_WW.1} and \eqref{eq_WW.2} together we obtain
\[
\limsup_{N\to\infty} \left|\frac{1}{|\Phi_N|} \sum_{g\in \Phi_N} f(T^g x) \phi(S^g y) -\int_X f \d \mu_X \int_Y\phi_k \d \mu_ Y \right|
\leq
\norm{\phi_k-\phi}{1}
\]
for every $k\in\N$.
Since $\phi_k\to\phi$ in $\lp^1$ we obtain the desired result.
\end{proof}

We now see how to derive the classical Wiener-Wintner theorem from \cref{thm:wienerWintner}.
\begin{corollary}[Wiener-Wintner theorem]
Let ${\mathbf X}=(X,T,\mu_X)$ be a $\Z$ system and let $f\in\lp^1(X,\mu_X)$.
There exists a set $X_0\subset X$ with $\mu(X_0)=1$ such that for every $\alpha\in\R$ and every $x\in X_0$ the limit
\begin{equation}\label{eq_WWclassical}
\lim_{N\to\infty} \frac{1}{N} \sum_{n=1}^Nf(T^nx)e(n\alpha)
\end{equation}
exists.
\end{corollary}
\begin{proof}
Denote by $\T$ the circle $\R/\Z$.
For each $\alpha\in\R$, let $R_\alpha:\T\to\T$ be the rotation $R_\alpha:t\mapsto t+\alpha$.
The pointwise ergodic theorem of Birkhoff applied to the system $(X\times\T,T\times R_\alpha)$ implies that there exists a set $X_\alpha\subset X$ with full measure such that, for every $x\in X_\alpha$, the limit \eqref{eq_WWclassical} exists.
The discrete spectrum $\Eig({\mathbf X})$ of $\mathbf{X}$ (cf.\ \cref{def:discrete-spectrum}) is at most countable, so the intersection $X_1=\bigcap_{\alpha\in\Eig({\mathbf X})}X_\alpha$ still has full measure.

Next let $X_2\subset X$ be the full measure set given by \cref{thm:wienerWintner} applied to ${\mathbf X}$ in place of ${\mathbf Y}$.
Since for every $\alpha\notin\Eig({\mathbf X})$ the systems ${\mathbf X}$ and $(\T,R_\alpha)$ are Kronecker disjoint, it follows that for every $x\in X_2$ the limit in \eqref{eq_WWclassical} exists. Therefore the limit exists for every $\alpha\in\R$ and every $x\in X_0=X_1\cap X_2$.
\end{proof}

\section{Applications to multicorrelation sequences and multiple recurrence}
\label{sec:multicorrelations}

In this section we present applications of our main results to the theory of multiple recurrence. In particular, this section contains proofs of Theorems \ref{thm:product-system-orthogonality} and \ref{thm:mrww-1}.
We remark that Theorems \ref{thm:product-system-orthogonality} and \ref{thm:mrww-1} only concern measure-preserving $\Z$ systems. The analogues for more general groups $G$ remain open. We formulate some open questions in this direction in \cref{sec:open-Q}.

\subsection{Preliminaries on nilmanifolds}
\label{sec:prelims-on-nil}

In \cite{Host_Kra05} Host and Kra established a structure theorem for multiple ergodic averages which revealed a deep connection between multi-correlation sequences and single-orbit dynamics on compact nilmanifolds.
In the proofs of Theorems \ref{thm:product-system-orthogonality} and \ref{thm:mrww-1} we make use of refinements of the Host-Kra structure theorem which appeared in \cite{Host_Kra09,Bergelson_Host_Kra05,Moreira_Richter18,Moreira_Richter21}.
The purpose of this subsection is to give an overview of these results and some related methods that we will use in the subsequent sections.

We begin with the definition of a nilmanifold.
A closed subgroup $\Gamma$ of $G$ is called \define{uniform} if $G/\Gamma$ is compact or, equivalently, if there exists a compact set $K$ such that $K\Gamma = G$.
Let $G$ be a $k$-step nilpotent Lie group and let $\Gamma\subset G$ be a uniform and discrete subgroup of $G$.
The quotient space $G/\Gamma$ is called a \define{nilmanifold}.
Naturally, $G$ acts continuously and transitively on $G/\Gamma$ via left-multiplication, that is $a(g\Gamma)=(ag)\Gamma$ for all $a\in G$ and all $g\Gamma\in G/\Gamma$.
On any nilmanifold $G/\Gamma$ there exists a unique $G$ invariant Borel probability measure called the \define{Haar measure of $G/\Gamma$}, which we denote by $\mu_{G/\Gamma}$ (cf.\ \cite[Lemma 1.4]{Raghunathan72}).
Given a fixed group element $a\in G$ the map $R:G/\Gamma\to G/\Gamma$ defined by $R(x)=ax$ for all $x=g\Gamma\in G/\Gamma$ is a \define{niltranslation} and the resulting  $\Z$ system $(G/\Gamma,R,\mu_{G/\Gamma})$ is called a \define{k-step nilsystem}.
We remark that $(G/\Gamma,R,\mu_{G/\Gamma})$ is ergodic if and only if $R$ acts transitively on $G/\Gamma$; as a matter of fact, if $R$ is transitive then $\mu_{G/\Gamma}$ is the unique $R$ invariant Borel probability measure on $G/\Gamma$ (cf.\ \cite{Auslander_Green_Hahn63,Parry69}).
Finally, for any $x\in G/\Gamma$ the orbit closure $Y=\overline{\{R^nx:n\in\Z\}}\subset G/\Gamma$ is a \define{sub-nilmanifold} of $G/\Gamma$, meaning that there exists a closed subgroup $H$ of $G$ such that $a\in H$, $Y=Hx$ and $\Lambda=H\cap \Gamma$ is a uniform and discrete subgroup of $H$.
In this case there exists a unique $H$ invariant Borel probability measure $\mu_{Y}$ on $Y$, called the Haar measure of the sub-nilmanifold $Y$, and the system $(Y,R, \mu_Y)$ is a nilsystem, as it is isomorphic to $(H/\Lambda,R,\mu_{H/\Lambda})$ (cf.\ \cite{Leibman06}).
For more information on nilmanifolds and nilsystems we refer the reader to \cite{Auslander_Green_Hahn63,Parry69,Parry70,Raghunathan72}.



{
\begin{theorem}[see {\cite[Revised Theorem 7.1]{Moreira_Richter21}}]\label{thm:3.4second}
Let $k\in\N$, let ${\mathbf X}=(G/\Gamma,R,\mu_{G/\Gamma})$ be an ergodic nilrotation and assume that $G/\Gamma$ is connected.
Define $S= R\times R^2\times \ldots\times R^k$ and for every $x\in G/\Gamma$ consider the sub-nilmanifold $\Omega({\mathbf X},x)$ of $(G/\Gamma)^{k}=G^k/\Gamma^{k}$ defined as
\[
\Omega({\mathbf X},x)=\overline{\big\{S^n(x,x,\ldots, x): n\in\Z\big\}}.
\]
Let $\mu_{\Omega({\mathbf X},x)}$ denote the Haar measure on $\Omega({\mathbf X},x)$ and let $\theta\in[0,1)$.
If 
$e(\theta)\notin \Eig(\mathbf{X})$ then for almost every $x\in G/\Gamma$, $e(\theta)\notin\Eig(\Omega({\mathbf X},x),S,\mu_{\Omega({\mathbf X},x)})$.
\end{theorem}
}

For our purposes we need a generalization of \cref{thm:3.4second} that holds for nilmanifolds $G/\Gamma$ that are not necessarily connected.


{
\begin{theorem}
\label{thm:3.4second-star}
Let $k\in\N$ and let ${\mathbf Z}=(G/\Gamma,R,\mu_{G/\Gamma})$ be an ergodic nilsystem. Define $S= R\times R^2\times \ldots\times R^k$ and
\[
\Omega({\mathbf Z},x)=\overline{\big\{S^n(x,x,\ldots, x): n\in\Z\big\}}
\subset (G/\Gamma)^{k}.
\]
For any $\theta\in[0,1)$, if 
$e(\theta)\notin \Eig({\mathbf Z})$ then for almost every $x\in G/\Gamma$, $e(\theta)\notin\Eig(\Omega({\mathbf Z},x),S,\mu_{\Omega({\mathbf Z},x)})$.
\end{theorem}
}

To derive \cref{thm:3.4second-star} from \cref{thm:3.4second} we need the following well-known lemma regarding nilsystems.

\begin{lemma}
[cf.\ \cite{Auslander_Green_Hahn63, Parry69,Leibman05}]
\label{thm:dynamics-nilrotation}
Suppose ${\mathbf Z}=(G/\Gamma,R,\mu_{G/\Gamma})$ is a nilsystem. Then the following are equivalent:
\begin{enumerate}	
[label=(\roman*),ref=(\roman*),leftmargin=*]
\item
$G/\Gamma$ is connected and ${\mathbf Z}$ is ergodic;
\item
${\mathbf Z}$ is totally ergodic.
\end{enumerate}
\end{lemma}

\begin{proof}[Proof of \cref{thm:3.4second-star}]
Suppose $G/\Gamma$ is not connected.
Since $G/\Gamma$ is compact, it splits into finitely many distinct connected components $Z_0,Z_1,\ldots,Z_{t-1}$.
It is also straightforward to show that $Z_i$ is itself a nilmanifold with Haar measure $\mu_{Z_i}$ and that $\mu_{G/\Gamma}=\frac{1}{t}(\mu_{Z_0}+\mu_{Z_1}+\ldots+\mu_{Z_{t-1}})$.
The ergodic niltranslation $R:G/\Gamma\to G/\Gamma$ cyclically permutes these connected components, so after re-indexing them if necessary we have $R^{tn+i}Z_0=Z_i$ for all $n\in\N$ and $i\in\{0,1,\ldots,t-1\}$.
In particular, for every $i\in\{0,\ldots,t-1\}$ the component $Z_i$ is $R^t$ invariant and $R^t:Z_i\to Z_i$ is an ergodic niltranslation on $Z_i$.
Let ${\mathbf Z}_i=(Z_i,R^t,\mu_{Z_i})$.

Since $Z_i$ is connected, it follows from \cref{thm:dynamics-nilrotation} that ${\mathbf Z}_i$ is totally ergodic. This means that $\Eig({\mathbf Z}_i)$ contains no roots of unity.
Also note that the function $\sum_{i=0}^{t-1} e\left(\frac{i}{t}\right)1_{Z_i}$ is an eigenfunction for $R$ with eigenvalue $e\left(\frac{i}{t}\right)$, where $e(x)=e^{2\pi i x}$ for all $x\in\R$. We conclude that
$$
\Eig({\mathbf Z})\cap\{\text{roots of unity}\}=\left\{1,e(\tfrac1t),e(\tfrac2t),\dots,e(\tfrac{t-1}t)\right\}.
$$
On the other hand, if $\zeta$ is not a root of unity then $\zeta$ is an eigenvalue for $R$ if and only if $\zeta^t$ is an eigenvalue for $R^t$; therefore
$$
\Eig({\mathbf Z})=\left(\Eig({\mathbf Z}_0)\right)^{\frac{1}{t}}\cdot \left\{1,e(\tfrac1t),e(\tfrac2t),\dots,e(\tfrac{t-1}t)\right\}.
$$

Let $x\in Z$ and let $i_0\in\{0,1,\ldots,t-1\}$ be such that $x\in Z_{i_0}$. Define \[\Omega({\mathbf Z}_{i_0},x)= \overline{\big\{S^{tn}(x,x,\ldots, x): n\in\Z\big\}}.\]
Observe that for $i\in\{0,\ldots,t-1\}$,
$$
S^i\Omega({\mathbf Z}_{i_0},x)\subset Z_{i_0+i}\times Z_{i_0+2i} \times \ldots \times Z_{i_0+ki}
$$
and hence $S^i\Omega({\mathbf Z}_{i_0},x)\cap S^j\Omega({\mathbf Z}_{i_0},x)=\emptyset$ for $i\neq j$.
Since ${\mathbf Z}_{i_0}$ is totally ergodic, it follows from \cref{thm:3.4second} that for almost every $x\in Z_{i_0}$ the nilsystem $\big(\Omega({\mathbf Z}_{i_0},x), S^t\big)$ is also totally ergodic.
In view of \cref{thm:dynamics-nilrotation} this means that $\Omega({\mathbf Z}_{i_0},x)$ is connected.
We deduce that for almost every $x\in Z$ the nilmanifold $\Omega({\mathbf Z},x)$ has $t$ connected components, because
$$
\Omega({\mathbf Z},x)=\bigcup_{i=0}^{t-1}S^i \Omega({\mathbf Z}_{i_0},x),
$$
where $\Omega({\mathbf Z}_{i_0},x), S^1\Omega({\mathbf Z}_{i_0},x),\ldots, S^{t-1}\Omega({\mathbf Z}_{i_0},x)$ are connected and distinct.

Finally, suppose $\theta\in[0,1)$ is such that $e(\theta)\notin\Eig({\mathbf Z})$.
Then $e(\theta)^t\notin\Eig({\mathbf Z}_i)$ for all $i$.
From \cref{thm:3.4second} we deduce that for almost every $x\in Z$, $e(\theta)^t\notin\Eig(\Omega({\mathbf Z}_{i_0},x),S^t,\mu_{\Omega({\mathbf Z}_{i_0},x)})$ (where, again, $i_0\in\{0,\dots,t-1\}$ is such that $x\in Z_{i_0}$) and hence $e(\theta)\notin\Eig(\Omega({\mathbf Z},x),S,\mu_{\Omega({\mathbf Z},x)})$.

\end{proof}

\begin{corollary}
\label{lem_due-to-erratum}
Let $k\in\N$ and let ${\mathbf Z}=(G/\Gamma,R,\mu_{G/\Gamma})$ be an ergodic nilsystem. Define $S= R\times R^2\times \ldots\times R^k$ and
\[
\Omega({\mathbf Z},x)=\overline{\big\{S^n(x,x,\ldots, x): n\in\Z\big\}}
\subset (G/\Gamma)^{k}. 
\]
Let $\mathbf{X}=(X,T,\mu_X)$ be a $\Z$ system.
If $\mathbf{X}$ and $\mathbf{Z}$ are Kronecker disjoint, then for almost every $x\in G/\Gamma$ the $\Z$ systems $\mathbf{X}$ and $(\Omega({\mathbf Z},x),S,\mu_{\Omega({\mathbf Z},x)})$ are Kronecker disjoint.
\end{corollary}

\begin{proof}
By way of contradiction, assume that there is a positive measure set $X'\subset  G/\Gamma$ such that $\mathbf{X}$ and $(\Omega({\mathbf Z},x),S,\mu_{\Omega({\mathbf Z},x)})$ are not Kronecker disjoint whenever $x\in X'$. 
This means that for any $x\in X'$ we can find some $\theta_x\in [0,1)$ such that $e(\theta_x)$ is a common eigenvalue for the systems $\mathbf{X}$ and $(\Omega({\mathbf Z},x),S,\mu_{\Omega({\mathbf Z},x)})$. Since $\mathbf{X}$ only possesses countably many eigenvalues, there exists a positive measure subset $X''\subset X'$ such that $\theta_x=\theta$ is constant for all $x\in X''$. 
Since $e(\theta)$ belongs to $\Eig(\Omega({\mathbf Z},x),S,\mu_{\Omega({\mathbf Z},x)})$ for all $x\in X''$ and $X''$ has positive measure, it follows from \cref{thm:3.4second-star} that $e(\theta)$ belongs to $\Eig(\mathbf{Z})$. This contradicts the assumption that $\mathbf{X}$ and $\mathbf{Z}$ are Kronecker disjoint.
\end{proof}

\subsection{Proofs of \cref{thm:product-system-orthogonality} and \cref{thm:mrww-1}}
\label{subsec:proofs-MR}

The following result, which will be used in the proofs of \cref{thm:product-system-orthogonality} and \cref{thm:mrww-1}, is contained implicitly in \cite[Subsection 7.3]{Host_Kra09}.

\begin{theorem}[cf.\ {\cite[Subsection 7.3]{Host_Kra09}}]
\label{thm:HK-7.3}
Let $k\in\N$, let $\mathbf{X}=(X,T,\mu_X)$ be a $\Z$ system and let $f_1,\ldots,f_k \in \lp^\infty(X,\mu_X)$.
Then for every $\epsilon>0$ there exists a $k$-step nilsystem $(G/\Gamma,R,\mu_{G/\Gamma})$, which is a factor of $(X,T,\mu_X)$, and there exist continuous functions $g_1,\ldots,g_k\in\cont(G/\Gamma)$ such that for every bounded complex-valued sequence $(a_n)_{n\in\N}$ one has
\[
\limsup_{N\to\infty} \bnorm{\frac{1}{N} \sum_{n=1}^N a_n\prod_{i=0}^k T^{in} f_i - \frac{1}{N} \sum_{n=1}^N a_n\prod_{i=0}^k (R^{in}g_i) \circ \pi}{2}
\le
\epsilon \sup_{n\in\N}{|a_n|}
\]
where $\pi:X\to G/\Gamma$ denotes the factor map from $(X,T,\mu_X)$ onto $(G/\Gamma,R,\mu_{G/\Gamma})$.
\end{theorem}

We will also need the following lemma.

\begin{lemma}
\label{lem:dons-1}
Let $(Y,S)$ be a topological $\Z$ system, let $\mu_Y$ be an ergodic $S$ invariant Borel probability measure on $Y$ and let $G\in \lp^1(Y,\mu_Y)$. Then there exists a set $Y'\subset Y$ with $\mu_Y(Y')=1$ such that for any ergodic nilsystem system $(G/\Gamma,\mu_{G/\Gamma},R)$ which is Kronecker disjoint from $(Y,S,\mu_Y)$, any $F\in \cont(G/\Gamma)$, any $x\in G/\Gamma$ and any $y\in Y'$ we have
$$
\lim_{N\to\infty}\frac1N\sum_{n=1}^N F(R^n x)G(S^n y)=\int_{G/\Gamma} F\d\mu_{G/\Gamma}\int_Y G\d\mu_Y.
$$
Moreover, if $(Y,S)$ is uniquely ergodic and $g\in\cont(Y)$ then we can take $Y'=Y$.
\end{lemma}

\begin{proof}
Since any nilsystem is measurably distal (see \cite[Theorem 2.14]{Leibman05}), \cref{lem:dons-1} is a special case of \cref{thm:pointwise-ET-with-distal-weights}.
\end{proof}

\begin{proof}[Proof of \cref{thm:product-system-orthogonality}]
Let $\mathbf{X}=(X,T,\mu_X)$ and $\mathbf{Y}=(Y,S,\mu_Y)$ be $\Z$ systems and assume $\mathbf{X}$ and $\mathbf{Y}$ are Kronecker disjoint.
Let $k,\ell\in\N$, $f_1,\dots,f_k\in \lp^\infty(X,\mu_X)$ and $g_1,\dots,g_\ell\in \lp^\infty(Y,\mu_Y)$.
According to \cite{Host_Kra05} and \cite{MR2257397} the limit
\[
F= \lim_{N \to \infty} \frac{1}{N} \sum_{n=1}^N\prod_{i=1}^k T^{in}f_i
\]
exists in $\lp^2(X,\mu_X)$ and the limit
\[
G= \lim_{N\to\infty} \frac{1}{N} \sum_{n=1}^N\prod_{j=1}^\ell S^{jn}g_j
\]
exists in $\lp^2(Y,\mu_Y)$.
Moreover, in view of \cite{Tao08} also the limit
\[
H = \lim_{N\to\infty} \frac{1}{N} \sum_{n=1}^N\prod_{i=1}^k \prod_{j=1}^\ell T^{in}f_i S^{jn}g_j
\]
exists in $\lp^2(X\times Y,\mu_X \otimes \mu_Y)$.
Our goal is to show that
\begin{equation}
\label{eq:6.1-0}
H=F\otimes G.
\end{equation}

We can assume without loss of generality that $\norm{f_i}{\infty} \le 1$ and $\norm{g_j}{\infty} \le 1$.
Fix $\epsilon > 0$.
First, we apply \cref{thm:HK-7.3} to find a $k$-step nilsystem $(G_X/\Gamma_X,R_X,\mu_{G_X/\Gamma_X})$, which is a factor of $(X,T,\mu_X)$, and a set of continuous functions $\tilde{f}_1,\ldots,\tilde{f}_k\in\cont(G_X/\Gamma_X)$ such that for every bounded complex-valued sequence $(a_n)_{n\in\N}$ one has
\[
\limsup_{N\to\infty} \bnorm{\frac{1}{N} \sum_{n=1}^N a_n \prod_{i=1}^k T^{in}f_i - \frac{1}{N} \sum_{n=1}^N a_n \prod_{i=1}^k  (R_X^{in}\tilde{f}_i)\circ \pi}{2}
\le
\epsilon \sup_{n\in\N}|a_n|
\]
where $\pi:X\to G_X/\Gamma_X$ denotes the factor map from $(X,T,\mu_X)$ onto $(G_X/\Gamma_X,R_X,\mu_{G_X/\Gamma_X})$.
In particular, if we choose $a_n=\prod_{j=1}^\ell S^{jn}g_j(y)$ as $y$ runs through $Y$, we obtain
\begin{equation}
\label{eq:6.1-1}
\sup_{y\in Y}
\limsup_{N\to\infty}
\bnorm{\frac{1}{N} \sum_{n=1}^N \prod_{i=1}^k \prod_{j=1}^\ell T^{in} f_i S^{jn}g_j(y) - \frac{1}{N} \sum_{n=1}^N \prod_{i=1}^k\prod_{j=1}^\ell \big( (R_X^{in}\tilde{f}_i)\circ\pi\big) S^{jn}g_j(y)}{2}
\le
\epsilon.
\end{equation}
From \eqref{eq:6.1-1} it follows that
\begin{equation}
\label{eq:6.1-2}
\limsup_{N\to\infty}
\bnorm{\frac{1}{N} \sum_{n=1}^N\prod_{i=1}^k \prod_{j=1}^\ell T^{in}f_i S^{jn}g_j - \frac{1}{N} \sum_{n=1}^N \prod_{i=1}^k \prod_{j=1}^\ell \big( (R_X^{in}\tilde{f}_i) \circ \pi \big) S^{jn}g_j}{2}
\le
\epsilon.
\end{equation}

Similarly, we can pick a $\ell$-step nilsystem $(G_Y/\Gamma_Y,R_Y,\mu_{G_Y/\Gamma_Y})$, which is a factor of $(Y,S,\mu_Y)$, and a set of continuous functions $\tilde{g}_1,\ldots,\tilde{g}_\ell\in\cont(G_Y/\Gamma_Y)$ such that for every bounded complex-valued sequence $(b_n)_{n\in\N}$ one has
\[
\limsup_{N\to\infty}
\bnorm{\frac{1}{N} \sum_{n=1}^N b_n \prod_{j=1}^\ell S^{jn}g_j - \frac{1}{N} \sum_{n=1}^N b_n \prod_{j=1}^\ell  (R_Y^{jn} \tilde{g}_j) \circ \eta }{2}
\le
\epsilon  \sup_{n\in\N}|b_n|,
\]
where $\eta:Y\to G_Y/\Gamma_Y$ denotes the factor map from $(Y,S,\mu_Y)$ onto $(G_Y/\Gamma_Y,R_Y,\mu_{G_Y/\Gamma_Y})$. If we set $b_n=\prod_{i=1}^k (R_X^{in}\tilde{f}_i)\circ\pi(x)$ then we get
\begin{equation}
\label{eq:6.1-3}
\sup_{x\in X}
\limsup_{N\to\infty}
\bnorm{\frac{1}{N} \sum_{n=1}^N \prod_{i=1}^k \prod_{j=1}^\ell (R_X^{in}\tilde{f}_i)(\pi x) S^{jn}g_j - \frac{1}{N} \sum_{n=1}^N \prod_{i=1}^k \prod_{j=1}^\ell (R_X^{in} \tilde{f}_i)(\pi x)\big((R_Y^{jn}\tilde{g}_j) \circ \eta\big)}{2}
\leq
\epsilon
\end{equation}
and hence
\begin{equation}
\label{eq:6.1-4}
\begin{split}
\limsup_{N\to\infty}
\bnorm{\frac{1}{N} \sum_{n=1}^N \prod_{i=1}^k \prod_{j=1}^\ell \big( (R_X^{in} \tilde{f}_i)\circ\pi \big) S^{jn}g_j
-
\frac{1}{N} \sum_{n=1}^N \prod_{i=1}^k \prod_{j=1}^\ell  \big((R_X^{in}\tilde{f}_i)\circ\pi\big)\big((R_Y^{jn}\tilde{g}_j)\circ \eta\big)}{2}
\le
\epsilon.
\end{split}
\end{equation}
Combining \eqref{eq:6.1-2} and \eqref{eq:6.1-4} yields
\begin{equation}
\label{eq:6.1-5}
\limsup_{N \to \infty}
\bnorm{\frac{1}{N} \sum_{n=1}^N \prod_{i=1}^k \prod_{j=1}^\ell T^{in} f_i S^{jn} g_j - \frac{1}{N} \sum_{n=1}^N \prod_{i=1}^k \prod_{j=1}^\ell \big((R_X^{in}\tilde{f}_i)\circ\pi \big)\big( (R_Y^{jn}\tilde{g}_j)\circ \eta\big) }{2}
\le
2\epsilon.
\end{equation}

Next, we claim that for almost every $x\in G_X/\Gamma_X$ and almost every $y\in G_Y/\Gamma_Y$ we have
\begin{equation}
\label{eq:6.1-6}
\lim_{N\to\infty}\Bigg|\frac1N\sum_{n=1}^N \prod_{i=1}^k\prod_{j=1}^\ell  \tilde{f}_i(R_X^{in}x)\tilde{g}_j(R_Y^{jn} y) -  \Bigg(\frac1N\sum_{n=1}^N \prod_{i=1}^k\tilde{f}_i(R_X^{in}x)\Bigg) \Bigg(\frac1N\sum_{n=1}^N \prod_{j=1}^\ell \tilde{g}_j(R_Y^{jn} y) \Bigg)\Bigg|=0.
\end{equation}
Assume for now that this claim holds.
It follows from \eqref{eq:6.1-1} and \eqref{eq:6.1-3} that
\begin{equation}
\label{eq:6.1-7}
\begin{split}
\limsup_{N\to\infty}
\bnbar
\Bigg( \frac{1}{N} \sum_{n=1}^N \prod_{i=1}^kR_X^{in}\tilde{f}_i\circ\pi\Bigg) & \Bigg(\frac{1}{N} \sum_{n=1}^N \prod_{j=1}^\ell R_Y^{jn}\tilde{g}_j \circ\eta\Bigg)
\\
&
-
\Bigg(\frac{1}{N} \sum_{n=1}^N \prod_{i=1}^k T^{in}f_i\Bigg) \Bigg(\frac{1}{N} \sum_{n=1}^N \prod_{j=1}^\ell S^{jn}g_j \Bigg)
\bnbar_2
\le
2\epsilon.
\end{split}
\end{equation}
Thus, combining \eqref{eq:6.1-7} with \eqref{eq:6.1-6} and \eqref{eq:6.1-5} gives
\[
\lim_{N\to\infty}
\bnorm{\frac{1}{N} \sum_{n=1}^N \prod_{i=1}^k\prod_{j=1}^\ell T^{in}f_i S^{jn}g_j - \left(\frac{1}{N} \sum_{n=1}^N\prod_{i=1}^k T^{in}f_i \right)
\cdot
\left( \frac{1}{N} \sum_{n=1}^N \prod_{j=1}^\ell S^{jn}g_j\right)}{2}
\leq
4\epsilon.
\]
Since $\epsilon>0$ was chosen arbitrarily, the proof of \eqref{eq:6.1-0} is complete.

It remains to show that \eqref{eq:6.1-6} is true.
Define $S_X= R_X\times R_X^2\times \ldots\times R_X^k$ and for every $x\in G_X/\Gamma_X$ consider $\Omega({\mathbf X},x)=\overline{\big\{S_X^n(x,x,\ldots, x): n\in\Z\big\}}$.
Also, define $F=\tilde{f}_1\otimes \ldots\otimes \tilde{f}_k$.
Similarly, we define $S_Y=R_Y\times R_Y^2\times \ldots\times R_Y^\ell$, $\Omega({\mathbf Y},y)=\overline{\big\{S_Y^n(y,y,\ldots, y): n\in\Z\big\}}$ and $G=\tilde{g}_1\otimes \ldots\otimes \tilde{g}_\ell$.
Hence \eqref{eq:6.1-6} can be rewritten as
\begin{equation}
\label{eq:6.1-8}
\begin{split}
\lim_{N\to\infty}\Bigg|\frac1N\sum_{n=1}^N F(S_X^n(x,\ldots,x)) & G(S_Y^n(y,\ldots,y))
\\
&-\Bigg(\frac1N\sum_{n=1}^N F(S_X^n(x,\ldots,x))\Bigg) \Bigg(\frac1N\sum_{n=1}^NG(S_Y^n(y,\ldots,y)) \Bigg)\Bigg|=0.
\end{split}
\end{equation}

Note that $(G_X/\Gamma_X,R_X,\mu_{G_X/\Gamma_X})$ and $(G_Y/\Gamma_Y,R_Y,\mu_{G_Y/\Gamma_Y})$ are Kronecker disjoint, because $(X,T,\mu_X)$ and $(Y,S,\mu_Y)$ are Kronecker disjoint.
In view of \cref{lem_due-to-erratum} it therefore follows that for $\mu_X$-almost every $x\in G_X/\Gamma_X $ and for $\mu_Y$-almost every $y\in G_Y/\Gamma_Y$ the two nilsystems $(\Omega({\mathbf X},x), S_X, \mu_{\Omega({\mathbf X},x)})$ and $(\Omega({\mathbf Y},y), S_Y, \mu_{\Omega({\mathbf Y},y)})$ are Kronecker disjoint.
We can now apply \cref{lem:dons-1} to conclude that for almost every $x\in G_X/\Gamma_X $ and almost every $y\in G_Y/\Gamma_Y$ we have
\begin{equation}
\label{eq:6.1-9}
\lim_{N\to\infty}\frac1N\sum_{n=1}^N F(S_X^n(x,\ldots,x))G(S_Y^n(y,\ldots,y))
=
\int_{\Omega({\mathbf X},x)} F\d\mu_{\Omega({\mathbf X},x)} \int_{\Omega({\mathbf Y},y)} G\d\mu_{\Omega({\mathbf Y},y)}.
\end{equation}
Since $(\Omega({\mathbf X},x), S_X)$ and $(\Omega({\mathbf Y},y), S_Y)$ are uniquely ergodic, we have that
\begin{eqnarray}
\label{eq:6.1-10}
\lim_{N\to\infty}\frac1N\sum_{n=1}^N F(S_X^n(x,\ldots,x))&=&\int_{\Omega({\mathbf X},x)} F\d\mu_{\Omega({\mathbf X},x)},
\\
\label{eq:6.1-11}
\lim_{N\to\infty}\frac1N\sum_{n=1}^N G(S_Y^n(x,\ldots,x))&=&\int_{\Omega({\mathbf Y},y)} G\d\mu_{\Omega({\mathbf Y},y)}.
\end{eqnarray}
Combining \eqref{eq:6.1-9} with \eqref{eq:6.1-10} and \eqref{eq:6.1-11} yields \eqref{eq:6.1-8}, which in turn implies \eqref{eq:6.1-6}. This completes the proof of \cref{thm:product-system-orthogonality}.
\end{proof}

\begin{proof}[Proof of \cref{thm:mrww-1}]
Let $(Y,S)$ be a topological $\Z$ system, let $\mu_Y$ be an $S$ invariant Borel probability measure on $Y$ and let $G\in \lp^1(Y,\mu_Y)$.
First we apply \cref{lem:dons-1} to find a set $Y'\subset Y$ with $\mu_Y(Y')=1$ such that for any ergodic nilsystem system $(G/\Gamma,\mu_{G/\Gamma},R)$ which is Kronecker disjoint from $(Y,S,\mu_Y)$, any $F\in \cont(G/\Gamma)$, any $x\in G/\Gamma$ and any $y\in Y'$ we have
$$
\lim_{N\to\infty}\frac1N\sum_{n=1}^N F(x)G(y)=\int_{G/\Gamma} F\d\mu_{G/\Gamma}\int_Y G\d\mu_Y.
$$
\cref{lem:dons-1} also guarantees that if $(Y,S)$ is uniquely ergodic and $G\in\cont(Y)$ then we can take $Y'=Y$.
Now let $(X,T,\mu_X)$ be a $\Z$ system that is Kronecker disjoint from $(Y,S,\mu_Y)$.
Fix $k\in\N$ and let $f_1,\dots,f_k\in \lp^\infty(X,\mu_X)$.
Our goal is to show that for any $y\in Y'$ we have
\begin{equation}
\label{eq:6.2-1}
\lim_{N\to\infty}\frac1N\sum_{n=1}^N G(S^ny)\prod_{i=1}^k T^{in}f_i=\left(\int_Y G\d\mu_Y\right)\cdot\left(\lim_{N\to\infty}\frac1N\sum_{n=1}^N\prod_{i=1}^k T^{in}f_i\right)
\end{equation}
in $\lp^2(X,\mu_X)$.

Note that \eqref{eq:6.2-1} is trivially true if $G$ is a constant function.
Hence, by replacing $G$ with $G-\int_Y G\d\mu_Y$ if necessary, we can assume without loss of generality that $\int_Y G\d\mu_Y=0$.
In this case \eqref{eq:6.2-1} reduces to
\begin{equation}
\label{eq:6.2-2}
\lim_{N\to\infty}\frac1N\sum_{n=1}^N G(S^ny)\prod_{i=1}^k T^{in}f_i=0.
\end{equation}
We can also assume without loss of generality that $\norm{G}{\infty} \le 1$ and that $\norm{f_i}{\infty} \le 1$ for all $i=1,\ldots,k$.

Fix $\epsilon > 0$.
We apply \cref{thm:HK-7.3} to find a $k$-step nilsystem $(G/\Gamma,R,\mu_{G/\Gamma})$, which is a factor of $(X,T,\mu_X)$, and a set of continuous functions $\tilde{f}_1,\ldots,\tilde{f}_k\in\cont(G/\Gamma)$ such that
\[
\limsup_{N \to \infty}
\bnorm{\frac{1}{N} \sum_{n=1}^N G(S^ny) \prod_{i=1}^k T^{in}f_i - \frac{1}{N} \sum_{n=1}^N G(S^ny) \prod_{i=1}^k  (R^{in}\tilde{f}_i)\circ \pi }{2}
\le
\epsilon,
\]
where $\pi:X\to G/\Gamma$ denotes the factor map from $(X,T,\mu_X)$ onto $(G/\Gamma,R,\mu_{G/\Gamma})$.
Therefore, to show \eqref{eq:6.2-2} it suffices to show that for almost every $x\in G/\Gamma$ one has
\begin{equation}
\label{eq:6.2-3}
\lim_{N\to\infty}\frac1N\sum_{n=1}^N G(S^ny) \prod_{i=1}^k  R^{in}\tilde{f}_i(x)=0.
\end{equation}
Define $S= R\times R^2\times \ldots\times R^k$, $Y_{x}=\overline{\big\{S^n(x,x,\ldots, x): n\in\Z\big\}}$ and also $F=\tilde{f}_1\otimes \ldots\otimes \tilde{f}_k$.
Clearly, \eqref{eq:6.2-3} is equivalent to
\begin{equation}
\label{eq:6.2-4}
\lim_{N\to\infty}\frac1N\sum_{n=1}^N G(S^ny) S^n F(x,\ldots, x)=0.
\end{equation}
It follows from \cref{lem_due-to-erratum} that for almost every $x\in G_X/\Gamma_X $ the two systems $(Y_{x}, S, \mu_{Y_{x}}, S)$ and $(Y,S,\mu_Y)$ are Kronecker disjoint.
Hence \cref{lem:dons-1} implies that for almost every $x\in G_X/\Gamma_X $ \eqref{eq:6.2-4} holds. This finishes the proof.
\end{proof}

\subsection{An example of multiple recurrence}

In this section we obtain an application of \cref{thm:main-result} to multiple recurrence.
\begin{definition}
\label{def:strongly-q-multiplicative}
Let $q$ be an integer $\geq 2$. A function $w:\N\to\C$ is called \define{strongly $q$-multiplicative} if
$w(n)=w(a_0)\cdots w(a_k)$, where
$$n=\sum_{i=0}^ka_iq^i\qquad0\leq a_i\leq q-1$$
is the base $q$ expansion of $n$.
For convenience we assume $w(0)=1$.
\end{definition}
The $\{-1,1\}$-valued Thue-Morse sequence is an example of a strongly $2$-multiplicative sequence, obtained by letting $w(0)=1$ and $w(1)=-1$.

When a strongly $q$-multiplicative function takes only finitely many values, say $w:\N\to{\mathcal A}\subset\C$, we can identify the function with a point (which by a slight abuse of notation we also denote by $w$) in the symbolic space ${\mathcal A}^\N$.
Let $S:{\mathcal A}^\N\to{\mathcal A}^\N$ be the usual shift map.
\begin{proposition}[\cite{MR883484}]
Let $w:\N\to{\mathcal A}\subset\C$ be a strongly $q$-multiplicative function. Then the orbit closure ${\mathbf Y}=\Big(\overline{\{S^nw:n\in\N\}},S\Big)$ is a uniquely ergodic system.
Moreover, the discrete spectrum $\Eig({\mathbf Y})$ with respect to the unique invariant measure $\mu_Y$ is contained in the set $\{e(a/q^n):a,n\in\N\}$.
\end{proposition}

Putting this proposition together with \cref{thm:mrww-1} we obtain the following corollary.
\begin{corollary}\label{cor_qmultiplicative}
  Let $w:\N\to\C$ be a strongly $q$-multiplicative function taking only finite many values and let ${\mathbf X}=(X,T,\mu_X)$ be a system with $\Eig({\mathbf X})\cap\{e(a/q^n):a,n\in\N\}=\{1\}$.
  For every $f_1,\dots,f_k\in\lp^\infty(X)$ we have
  $$\lim_{N\to\infty}\frac1N\sum_{n=1}^Nw(n)\prod_{i=1}^kT^{in}f_i= \left(\lim_{N\to\infty}\frac1N\sum_{n=1}^Nw(n)\right)\cdot\left( \lim_{N\to\infty}\frac1N\sum_{n=1}^N\prod_{i=1}^kT^{in}f_i\right).$$
\end{corollary}

We can use this to derive a multiple recurrence result for level sets of strongly $q$-multiplicative functions.
\begin{theorem}
  Let $m,q\in\N$, let ${\mathcal A}\subset\C$ be the set of $m$-th roots of $1$ and let $w:\N\to{\mathcal A}$ be a strongly $q$-multiplicative function.
  For every $z\in{\mathcal A}$, the level set $R=\{n\in\N:w(n)=z\}$ either has $0$ density or satisfies the following multiple recurrence property:
Let ${\mathbf X}=(X,T,\mu_X)$ be a $\Z$ system and assume that $\Eig({\mathbf X})$ contains no non-trivial $q$-th root of $1$.
  Then for every $A\subset X$ with $\mu(A)>0$ and every $k\in\N$ there exists $n\in R$ such that
  $$\mu(A\cap T^{-n}A\cap T^{-2n}A\cap\cdots\cap T^{-kn}A)>0.$$
\end{theorem}

\begin{proof}
The indicator function $1_R(n)$ of $R=\{n\in\N:w(n)=z\}$ can be expressed as $\delta_z\circ w(n)$, where $\delta_z:{\mathcal A}\to\C$ is the function $\delta_z(u)=1$ if $u=z$ and $\delta_z(u)=0$ otherwise.
The space of functions from $\mathcal A$ to $\C$ is a vector space spanned by the functions $u\mapsto u^j$, $j=0,\dots,m-1$ and in particular $\delta_z$ is a linear combination of the functions $u^j$.
It follows that $1_R(n)$ is a linear combination of the functions $w^j(n)$ for $j=0,\dots,m-1$.

Observe that each power $w^j$ of $w$ is a strongly $q$-multiplicative function, and thus \cref{cor_qmultiplicative} applied to the functions $f_i=1_A$ implies that
$$\lim_{N\to\infty}\frac1N\sum_{n=1}^Nw^j(n)\mu\left(\bigcap_{i=0}^kT^{-in}A\right) =\left(\lim_{N\to\infty}\frac1N\sum_{n=1}^Nw^j(n)\right)\cdot\left( \lim_{N\to\infty}\frac1N\sum_{n=1}^N\mu\left(\bigcap_{i=0}^kT^{-in}A\right) \right).$$
By linearity, this implies that

$$\lim_{N\to\infty}\frac1N\sum_{n=1}^N1_R(n)\mu\left(\bigcap_{i=0}^kT^{-in}A\right) =\left(\lim_{N\to\infty}\frac1N\sum_{n=1}^N1_R(n)\right)\cdot\left( \lim_{N\to\infty}\frac1N\sum_{n=1}^N\mu\left(\bigcap_{i=0}^kT^{-in}A\right) \right).$$
If $R$ has positive density, then the first factor in the right hand side of the previous equation is positive. The fact that the second factor is also positive is the content of Furstenberg's multiple recurrence theorem \cite[Theorem 11.13]{Furstenberg77}.
Therefore the left hand side has to be positive as well and this implies the desired conclusion.
\end{proof}

\section{Some open questions}
\label{sec:open-Q}

It is tempting to define $\mathbf{X}$ and $\mathbf{Z}$ to be quasi-disjoint if the natural map from $\ergjoinings(\mathbf{X},\mathbf{Z})$ to $\ergjoinings(\kron \mathbf{X},\kron \mathbf{Z})$ is a bijection.
However, \cite[Example~2]{Berg71} shows this is in fact a strictly stronger notion than quasi-disjointness, which amounts to requiring that the above map be a bijection almost everywhere with respect to the measure on $\ergjoinings(\mathbf{X},\mathbf{Z})$ given by the ergodic decomposition of the product measure.
In light of this we ask the following question.

\begin{question}
Is it true that a system $\mathbf{X}$ is quasi-disjoint from $\mathbf{Y}$ if and only if the support of the measure appearing in the ergodic decomposition of $\mu\times\nu$ equals $\ergjoinings(\mathbf{X},\mathbf{Y})$.
\end{question}

The notion of disjointness can be described in terms of factor maps.
We ask if a similar characterization of quasi-disjointness is possible:

\begin{question}
Is is true that ${\mathbf X}$ and ${\mathbf Y}$ are quasi-disjoint if and only if any system ${\mathbf Z}$ which has ${\mathbf X}$, ${\mathbf Y}$ and $\kron{\mathbf X}\times\kron{\mathbf Y}$ as factors, also has ${\mathbf X}\times{\mathbf Y}$ as a factor?
\end{question}

We are also interested in the following potential extensions of our theorem and its applications.

\begin{question}
Is a system ${\mathbf X}$ quasi-disjoint from any ergodic system if and only if it is quasi-disjoint from itself?
\end{question}


We expect the following question, which seeks a generalization of Theorem~\ref{thm:product-system-orthogonality} to countable, amenable groups, to be quite difficult.

\begin{question}
Fix a countable, amenable group $G$ with a \Folner{} sequence $\Phi$.
Let $\mathbf{X} = (X,T,\mu_X)$ be an ergodic $G^k$ system and let $\mathbf{Y} = (Y,S,\mu_Y)$ be an ergodic $G^\ell$ system.
Given $1 \le i \le j \le k$ write $T_{[i,j]}$ for the $G$ action induced by the inclusion of $G$ in $G^k$ diagonally on the coordinates $i,\dots,j$ and similarly for $S_{[i,j]}$ with $1 \le i \le j \le \ell$.
Under what conditions on $\mathbf{X}$ and $\mathbf{Y}$ do we have
\[
\lim_{N \to \infty} \frac{1}{|\Phi_N|} \sum_{g \in \Phi_N} \prod_{i=1}^k T_{[1,i]}^g f_i \prod_{j=1}^\ell S_{[1,j]}^g h_j
=
\left( \lim_{N \to \infty} \frac{1}{|\Phi_N|} \sum_{g \in \Phi_N} \prod_{i=1}^k T_{[1,i]}^g f_i\right)\!\!\!
\left( \lim_{N \to \infty} \frac{1}{|\Phi_N|} \sum_{g \in \Phi_N} \prod_{j=1}^\ell S_{[1,j]}^g h_j \right)
\]
in $\lp^2(X \times Y, \mu_X \otimes \mu_Y)$ for all $f_1,\dots,f_k$ in $\lp^\infty(X,\mu_X)$ and all $h_1,\dots,h_\ell$ in $\lp^\infty(Y,\mu_Y)$?
\end{question}

\printbibliography

\bigskip
\footnotesize

\noindent
Joel Moreira\\
\textsc{University of Warwick}\par\nopagebreak
\noindent
\href{mailto:joel.moreira@warwick.ac.uk}
{\texttt{joel.moreira@warwick.ac.uk}}
\\

\noindent
Florian K.\ Richter\\
\textsc{École Polytechnique Fédérale de Lausanne (EPFL)}\\
\noindent
\href{mailto:f.richter@epfl.ch}
{\texttt{f.richter@epfl.ch}}
\\

\noindent
Donald Robertson\\
\textsc{University of Manchester}\par\nopagebreak
\noindent
\href{mailto:donald.robertson@manchester.ac.uk}{\texttt{donald.robertson@manchester.ac.uk}}

\end{document}